\documentclass[11pt]{article}

\usepackage{mathpazo}
\usepackage{amsfonts}
\usepackage{amsmath}
\usepackage{graphicx}
\usepackage{latexsym}
\usepackage{mathabx,multicol}
\usepackage[margin=1.05in]{geometry}
  
\usepackage[T1]{fontenc}
\usepackage{times}
\usepackage{color,graphicx}
\usepackage{array}
\usepackage{enumerate}
\usepackage{amsmath}
\usepackage{amssymb}
\usepackage{amsthm}
\usepackage{pgfplots}
\usepackage{pgf}
\usepackage{tikz}
\usetikzlibrary{patterns}
\usepgfplotslibrary{patchplots} 
\usetikzlibrary{pgfplots.patchplots} 
\pgfplotsset{width=9cm,compat=1.5.1}

\usepackage{amsthm}

\newtheorem{definition}{Definition}
\newtheorem{lemma}{Lemma}
\newtheorem{theorem}{Theorem}
\newtheorem{corollary}{Corollary}

\newtheorem{proposition}{Proposition}

\newtheorem{example}{Example}

\newcommand{\be}{\mathbf e}
\newcommand{\bv}{\mathbf v}

\newcommand{\bu}{\mathbf u}
\newcommand{\bx}{\mathbf x}
\newcommand{\by}{\mathbf y}
\newcommand{\bz}{\mathbf z}
\newcommand{\bmu}{{\boldsymbol{\mu}}}
\newcommand{\one}{\mathbf 1}
\newcommand{\zero}{\mathbf 0}
\def\ii{\mathrm{i}}
\newcommand\spec{{\mathrm{spec}}}
\newcommand{\supp}{{\Phi}}

\usepackage{xcolor}
\usepackage{makeidx}
\usepackage[colorlinks=true,linkcolor=blue,anchorcolor=blue,citecolor=red,urlcolor=magenta]{hyperref}
\usepackage[alphabetic,backrefs]{amsrefs}

\def\C{\mathbb{C}}
\def\Q{\mathbb{Q}}
\def\R{\mathbb{R}}
\def\Z{\mathbb{Z}}

\begin{document}
\title{Local $\epsilon$-uniform mixing in continuous quantum walks}

\author{
	Hermie Monterde\textsuperscript{\hspace{-0.05cm} \!1}
}
\date{}
\maketitle

\begin{abstract}
Let $X$ be a weighted graph and $M$ be its adjacency, Laplacian or signless Laplacian matrix. In a continuous quantum walk on $X$, local $\epsilon$-uniform mixing occurs at vertex $u$ if the $u$th column of the matrix $U(t)=e^{\ii tM}$ can be made arbitrarily close to a vector whose all entries have equal magnitude. Using the spectral and combinatorial properties of $X$, we derive necessary conditions for local $\epsilon$-uniform mixing to occur in $X$. This includes an inequality involving all entries of each eigenvector of $M$, as well as an upper bound on the degree of vertex $u$ when $M$ is the Laplacian or signless Laplacian matrix. We use these necessary conditions to rule out local $\epsilon$-uniform mixing in numerous classes of graphs, most of which are non-regular. We also show that almost all planar graphs (resp., trees) contain a vertex that does not admit local $\epsilon$-uniform mixing for any assignment of edge weights. Furthermore, we prove if $X$ has $n$ vertices and admits local $\epsilon$-uniform mixing at a vertex contained in a subgraph with a twin, then the number of vertices of this twin subgraph must be at least $\sqrt{n}$. In particular, we establish that a graph on $n\geq 5$ vertices does not admit local $\epsilon$-uniform mixing at a vertex with a twin. 
\end{abstract}

	

\addtocounter{footnote}{1}
\footnotetext{Department of Mathematics and Statistics, University of Regina, SK, Canada S4S 0A2, Hermie.Monterde@uregina.ca}

\section{Introduction}\label{secINTRO}

Let $X$ be a weighted graph, and $M$ be its adjacency, Laplacian or signless Laplacian matrix. A \textit{(continuous) quantum walk} on $X$ (relative to $M$) is determined by a one-parameter family of symmetric unitary matrices:
\begin{equation*}
\label{U}
U_M(t)=e^{\ii tM},\quad t\in\R.
\end{equation*}
A complex vector (or a matrix) is \textit{uniform} if all its entries have equal magnitude. We say that \textit{uniform mixing} occurs in $X$ at time $\tau$ if $U_M(\tau)$ is a uniform matrix. A relaxation of uniform mixing is \textit{$\epsilon$-uniform mixing}, which occurs in $X$ if all entries of $U_M(t)$ can be made arbitrarily close to a uniform matrix. We say that
\textit{local $\epsilon$-uniform mixing} occurs at vertex $u$ in $X$ if the $u$th column of $U_M(t)$ can be made arbitrarily close to a uniform vector. Note that the existence of local uniform mixing  implies local $\epsilon$-uniform mixing, and the existence of $\epsilon$-uniform mixing requires each vertex in $X$ to exhibit local $\epsilon$-uniform mixing. 

The first family of graphs known to exhibit uniform mixing is the hypercubes \cite{moore2002quantum}, and the second is the complete graphs $K_q$, where $q\in\{2,3,4\}$ \cite{Ahmadi}. In \cite{best2008mixing}, Best et al.\ proved that Cartesian products preserve uniform mixing at the same time, and concluded that the Hamming graphs $H(n,q)$ have uniform mixing if and only if $q\in\{2,3,4\}$. In \cite{Godsil2013}, Godsil, Mullin, and Roy characterized strongly regular graphs that admit uniform mixing. They also showed that uniform mixing on a bipartite (resp., regular) graph with $n$ vertices, respectively, implies that $n\equiv 0$ (mod 4) (resp., $n$ is sum of two integer squares). Cycles of even and prime lengths were also shown to not admit uniform mixing. In \cite{Godsil2015UniformMO}, Godsil and Zhan discovered families of Cayley graphs over $\Z_q^d$ with $q\in\{2,3,4\}$ that admit uniform mixing. They also provided the first infinite family of non-regular graphs that admit uniform mixing, which consists of Cartesian powers of the star on four vertices. Later, Chan found infinite families of graphs in the Hamming scheme that admit uniform mixing, and characterized uniform mixing in folded $n$-cubes, halved $n$-cubes, and folded halved $n$-cubes \cite{Chan2013}. Most recently, a family of oriented, normal, and nonabelian Cayley graphs is presented in \cite{sin2025uniform}. For $\epsilon$-uniform mixing, no family is known apart from cycles of prime lengths \cite{Godsil2013}. For local uniform mixing, Godsil showed that the apex of a cone on a $d$-regular graph admits this property \cite{SedQW} when $d\in\{0,1,2\}$. From this short survey, we see that only a handful of graphs are known to admit uniform mixing and most of them are regular. Part of the reason for this short list is the absence of characterization of uniform mixing. In this paper, we use the spectral properties of $M$ and the combinatorial properties of $X$ to derive necessary conditions for local $\epsilon$-uniform mixing at vertex $u$ of $X$, which in turn provides necessary conditions for local uniform mixing, $\epsilon$-uniform mixing and uniform mixing. We then use these conditions to rule out local $\epsilon$-uniform mixing in numerous classes of graphs, most of which are non-regular. 

This paper is organized as follows. In Section~\ref{sec:UM}, we give equivalent characterizations of local $\epsilon$-uniform mixing (Proposition~\ref{proplum}), local uniform mixing (Corollary~\ref{charlum}), and $\epsilon$-uniform mixing (Proposition~\ref{Lem:compHad}). Using these characterizations, we derive necessary conditions for local $\epsilon$-uniform mixing in Section~\ref{sec:neccon}, which include an upper bound on the degree of vertex $u$ in terms of the average degree of $X$ when $M$ is the Laplacian or the signless Laplacian matrix of $X$ (Theorem~\ref{Cor:avedeg}), as well as an inequality that involves all entries of each eigenvector of $X$ (Theorem~\ref{lum}). In Section~\ref{sec:tw}, we define true and false twin subgraphs. We prove that 
the existence of small twin subgraphs is forbidden amongst large graphs admitting $\epsilon$-uniform mixing (Corollary~\ref{corintro}). We also show that a graph on $n\geq 5$ vertices does not admit local $\epsilon$-uniform mixing at a vertex with a twin (Corollary~\ref{twin}). Section~\ref{sec:per} is a discussion on the connection of local uniform mixing and periodicity. Section~\ref{sec:bip} is devoted to local $\epsilon$-uniform mixing in bipartite graphs relative to the adjacency matrix. We establish parity requirements on $n$ and $\operatorname{deg}u$ whenever $u$ admits local $\epsilon$-uniform mixing in an unweighted bipartite graph $X$ (Corollary~\ref{nodd}). We also prove that if a singular integer-weighted bipartite graph admits local $\epsilon$-uniform mixing at vertex $u$, then $n$ must be a perfect square; and if, in addition, the nullspace of $A$ contains a $\{-1,0,1\}$ vector with a nonzero $u$th entry, then the size of the part containing $u$ is at least $\sqrt{n}$ and has the same parity as $n$ (Corollary~\ref{cor:singularbip}). Thus, weighted bipartite graphs with $\epsilon$-uniform mixing cannot be `too' unbalanced (Theorem~\ref{thm:eumbip}). We also prove a rarity result for integer-weighted bipartite graphs that admit local uniform mixing (Theorem~\ref{raregenA}). Section~\ref{sec:plan} deals with planar graphs. We show that almost all connected planar graphs (resp., trees) contain a vertex that does not admit local $\epsilon$-uniform mixing for any assignment of edge weights (Theorem~\ref{lumplntr}). We also provide upper bounds on the degree of vertices admitting local $\epsilon$-uniform mixing in various classes of planar graphs. We investigate local $\epsilon$-uniform mixing on trees and unicyclic graphs in Section~\ref{sec:tu} relative to the adjacency matrix. We show that if $X$ is a tree or a bipartite unicyclic graph that is not cycle, then local $\epsilon$-uniform mixing does not occur at a pendent vertex (Corollary~\ref{treeuni}). It also turns out that $K_{1,3}$ is the only unweighted tree with no degree-two vertex that admits $\epsilon$-uniform mixing (Corollary~\ref{k13}), and that the subdivision of any weighted tree does not admit $\epsilon$-uniform mixing (Theorem~\ref{thm:subdtree}). Finally, open questions are presented in Section~\ref{sec:oq}. The remainder of this section is alloted to basic definitions and notation relevant to the paper.

Throughout this paper, we assume that $X$ is a simple (loopless) connected weighted undirected graph. We denote the vertex and edge sets of $X$ by $V(X)$ and $E(X)$, respectively. Unless otherwise stated, all edges are assumed to have positive weights. We say that $X$ is \textit{unweighted} if all edges of $X$ have weight one. For $u\in V(X)$, we denote the characteristic vector of $u$ as $\be_u$. The all-ones vector, the zero vector, and the identity matrix of appropriate sizes are denoted by $\one$, $\zero$, and $I$, respectively. We write the transpose of a matrix $M$ by $M^T$, and the conjugate of $M$ by $\overline{M}$. For matrices $A$ and $B$ of the same size, we denote their Schur product by $A\circ B$. For a vector $\bx\in\C^n$, we let $\|\bx\|=\overline{\bx}^T\bx$. We denote the simple unweighted empty, cycle, complete, and path graphs on $n$ vertices by $O_n$, $C_n$, $K_n$, and $P_n$, respectively. We denote the complete bipartite graph by $K_{m,n}$, where $m$ and $n$ are the sizes of the parts. The adjacency matrix $A(X)$ of $X$ is the matrix indexed by $V(X)$ whose $(u,v)$-entry is equal to the weight of the edge $\{u,v\}$ whenever $u$ and $v$ are adjacent, and 0 otherwise. The Laplacian matrix  and signless Laplacian matrix of $X$ are defined as $L(X)=D(X)-A(X)$ and $Q(X)=D(X)+A(X)$, respectively, where $D(X)$ is the diagonal matrix of vertex degrees of $X$. We use $M(X)$ to denote $A(X)$, $L(X)$ or $Q(X)$. If the context is clear, then we simply write these matrices as $M$, $A$, $L$ and $Q$, respectively.

The matrices $U_M(t)$ and $M$ in (\ref{U}) are the \textit{transition matrix} and \textit{Hamiltonian} of the quantum walk, respectively. We write $U_M(t)$ as $U(t)$ if $M$ is clear from the context. Typically, $M$ is taken to be $A$, $L$, or $Q$, but in general, any real symmetric matrix $M$ that respects the adjacencies of $X$ works (that is, $M_{u,v}=0$ if and only if there is no edge between $u$ and $v$). Let $\spec(M)$ denote the distinct eigenvalues of $M$. Since $M$ is real symmetric, we can write $M=\sum_{\lambda\in\spec(M)}\lambda E_{\lambda}$ in its spectral decomposition,  where $E_{\lambda}$ is the orthogonal projection matrix onto the eigenspace for $\lambda\in\spec(M)$.
Thus, we may write
\begin{equation*}
U(t)=\sum_{\lambda\in\spec(M)}e^{\ii t\lambda}E_{\lambda}.
\end{equation*}
A vector $\bx\in\C^{|V(X)|}$ is \textit{periodic} in $X$ (relative to $M$) if there is a time $\tau>0$ such that $U(\tau)\bx=\gamma\bx$, where $\gamma\in\C$. Vertex $u$ in $X$ is \textit{periodic} if the vector $\be_u$ is periodic. The graph $X$ is \textit{periodic} if every vertex in $X$ is periodic at the same time. Since $U(t)$ is a unitary matrix, we get $\sum_{j\in V(X)}|U(t)_{u,j}|^2=1$ for all $t\in\R$ and for all $u\in V(X)$. Thus, every column (or row) of $U(t)\circ\overline{U(t)}$ determines a probability distribution.

\section{Uniform mixing}\label{sec:UM} 

Let $X$ be a weighted graph on $n$ vertices. A vector $\bx\in\C^n$ is $\sigma$-\textit{uniform} if the moduli of all entries in $\bx$ are equal to $\sigma$. 
If column $u$ of $U(\tau)$ is $\frac{1}{\sqrt{n}}$-uniform for some $\tau>0$, then \textit{local uniform mixing} occurs at vertex $u$ at time $\tau$ relative to $M$ (equivalently, the probability distribution determined by the $u$th column of $U(t)\circ\overline{U(t)}$ is uniform). If local uniform mixing occurs at each vertex of $X$ at time $\tau$, then $X$ admits \textit{uniform mixing} at time $\tau$ relative to $M$. A relaxation of local uniform mixing is \textit{local  $\epsilon$-uniform mixing} at vertex $u$, which occurs if for each $\epsilon>0$, there is a time $\tau$ such that 
\begin{center}
$\left\|\big(U(\tau)\be_u\big)\circ\big(\overline{U(\tau)}\be_u\big)-(1/n)\one\right\|<\epsilon.$
\end{center}
If for each $\epsilon>0$, there is a $\tau$ such that the above holds for each $u\in V(X)$, then $X$ admits \textit{$\epsilon$-uniform mixing}  relative to $M$. Thus, a vertex admitting local uniform mixing also admits local $\epsilon$-uniform mixing. Similarly, a graph admitting uniform mixing also admits $\epsilon$-uniform mixing. However, the converses of the two preceding statements need not be true. Moreover, if $u$ and $v$ are similar, i.e., there is an automorphism that sends $u$ to $v$, then  local $\epsilon$-uniform mixing occurs (resp, local uniform mixing) at $u$ if and only if it occurs at $v$ at the same time. Thus, $\epsilon$-uniform mixing (resp., uniform mixing) occurs in a vertex-transitive graph if and only if local $\epsilon$- uniform mixing (resp, local uniform mixing) occurs at one vertex in the graph.

Suppose local uniform mixing occurs at vertex $u$ in $X$ at time $\tau$. As we will show below, this is equivalent to the fact that $U(\tau)\be_u=\frac{1}{\sqrt{n}}\sum_{j=1}^{n}s_j\be_j$, where each $s_j$ is a unit complex number. In other words, local uniform mixing at vertex $u$ in $X$ is a unitary mapping of $\be_u$ to a uniform superposition of all characteristic vectors of vertices in $X$. Adapting the physical interpretation of $s$-pair state transfer in \cite{kim2024generalization}, we get that local uniform mixing represents the maximal entanglement of all qubits in the network represented by $X$ at a later time when evolved from a vertex state. In this case,  $U(\tau)\be_u$ is viewed as a pure state in the 1-excitation subspace $\mathbb{C}^n$ of the full $2^n$-dimensional system of $n$ spins. 

Our goal in this section is to give equivalent characterizations of local $\epsilon$-uniform mixing. To do this,
we define the \textit{eigenvalue support} of a vector $\bx$, denoted $\supp_{\bx}$, as the set $\supp_{\bx}=\{\lambda\in\spec(M):E_{\lambda}\bx\neq\zero\}.$

\begin{proposition}
\label{proplum}
Let $X$ be a weighted graph on $n$ vertices. The following are equivalent.
\begin{enumerate}
\item Local $\epsilon$-uniform mixing occurs at $u$.
\item There is a sequence $\{\tau_m\}\subseteq\R$ and a 1-uniform vector $\bmu\in\C^n$ such that
\begin{equation}
\label{locum}
\lim_{m\rightarrow\infty}\sqrt{n}U(\tau_m)\be_u=\bmu.
\end{equation}
\item There is a sequence $\{\tau_m\}\subseteq\R$ and a 1-uniform vector $\bmu\in\C^n$ such that 
\begin{center}
$\big(\displaystyle\lim_{m\rightarrow\infty}e^{\ii \tau_m\lambda}\big)\sqrt{n}E_\lambda\be_u=E_\lambda\bmu\quad $ for all $\lambda\in\supp_{\be_u}$.
\end{center}
Equivalently, for each eigenpair $(\bv,\lambda)$ with $\lambda\in\supp_{\be_u}$, we have $\big(\displaystyle\lim_{m\rightarrow\infty}e^{\ii\tau_m\lambda}\big)\sqrt{n}(\bv^T\be_u)=(\bv^T\bmu)$. Thus, $\displaystyle\lim_{m\rightarrow\infty}e^{\ii\tau_m\lambda}$ exists for all  $\lambda\in\supp_{\be_u}$.
\end{enumerate}
\end{proposition}

\begin{proof}
Suppose local $\epsilon$-uniform mixing at vertex $u$. Then there exists $\{\tau_m\}\subseteq\R$ such that
\begin{center}
$\displaystyle\lim_{m\rightarrow\infty}\left(U(\tau_m)\be_u\circ \overline{U(\tau_m)}\be_u\right)=(1/n)\one.$
\end{center}
Hence, $U(\tau_m)\be_u$ is a bounded sequence in $\C^{n}$, and so there is a subsequence $\{\tau_{m_\ell}\}$ of $\{\tau_m\}$ such that $U(\tau_{m_\ell})\be_u$ converges to some $\by\in\C^n$. Thus, $\by\circ\overline{\by}=\frac{1}{n}\one$, and so $\by=\frac{1}{\sqrt{n}}\bmu$ for some 1-uniform vector $\bmu\in\C^n$. This proves that $(1)$ implies $ (2)$. The converse follows by definition. The equivalence of $(2)$ and $(3)$ is immediate from the spectral decomposition of $M$. 
\end{proof}

If Proposition~\ref{proplum}(2) holds, then we say that \textit{local $\epsilon$-uniform mixing occurs at $u$ relative to $\{\tau_m\}$}. 

The following is immediate from Proposition~\ref{proplum}.

\begin{corollary}
\label{charlum}
Let $X$ be a weighted graph on $n$ vertices and $M=\sum_{j}e^{\ii t\lambda_j}E_j$ be the spectral decomposition of $M$. The following are equivalent.
\begin{enumerate}
\item local uniform mixing occurs at $u$ at time $\tau$.
\item There exists a 1-uniform vector $\bmu\in\C^n$ such that $\sqrt{n}U(\tau)\be_u=\bmu$.
\item There exists a 1-uniform $\bmu\in\C^n$ such that $\sqrt{n}e^{\ii \tau\lambda}E_\lambda\be_u=E_\lambda\bmu$ for all $\lambda\in\supp_{\be_u}$.
Equivalently, for each eigenpair $(\bv,\lambda)$ with $\lambda\in\supp_{\be_u}$, $\sqrt{n}e^{\ii\tau\lambda}(\bv^T\be_u)=(\bv^T\bmu)$.
\end{enumerate}
\end{corollary}

An $n\times n$ matrix $\mathcal{H}$ is a \textit{complex Hadamard} matrix if all its entries are unit complex numbers, and 
$
\mathcal{H}\overline{\mathcal{H}}^T=nI.$ If the entries of a complex Hadamard matrix $\mathcal{H}$ are $r$th roots of unity for some integer $r$, then $\mathcal{H}$ is a \textit{Butson Hadamard} matrix. If the entries of $\mathcal{H}$ belong to the set $\{\pm 1,\pm\ii\}$ (resp., $\{\pm 1\}$), then $\mathcal{H}$ is a \textit{Turyn Hadamard} matrix (resp., a \textit{real Hadamard} matrix). A complex Hadamard matrix is \textit{dephased} if its first row and column are all ones, and \textit{non-dephased} otherwise. 

\begin{proposition}
\label{Lem:compHad}
We have $\epsilon$-uniform mixing in $X$ if and only if there exists a sequence $\{\tau_m\}\subseteq \R$ such that
\begin{center}
$\displaystyle\lim_{m\rightarrow\infty}\sqrt{n}U(\tau_m)=\mathcal{H}$,
\end{center}
where $\mathcal{H}$ is a symmetric complex Hadamard matrix. Additionally, if $M\in\{A,L\}$, then $\mathcal{H}$ is non-dephased.
\end{proposition}

\begin{proof}
By Proposition~\ref{proplum}, $\displaystyle\lim_{m\rightarrow\infty}\sqrt{n}U(\tau_m)=\mathcal{H}$, where all entries of $\mathcal{H}$ have unit moduli. Since
\begin{center}
$\displaystyle\mathcal{H}\overline{\mathcal{H}}^T=n\left(\lim_{m\rightarrow\infty}U(\tau_m)\right)\left(\overline{\lim_{m\rightarrow\infty}U(\tau_m)}\right)^T=n\left(\lim_{m\rightarrow\infty}U(\tau_m)\overline{U(\tau_m)}^T\right)=nI,$
\end{center}
it follows that $\mathcal{H}$ is a complex Hadamard matrix. Finally, if $\bmu$ is a column of $\mathcal{H}$, then a result in a latter section (Corollary~\ref{no1}) implies that $\bmu\neq \pm\one$, and so $\mathcal{H}$ is non-dephased.
\end{proof}

\begin{corollary}
\label{Lem:compHad1}
Uniform mixing occurs in $X$ at time $\tau$ if and only if  $\sqrt{n}U(\tau)$
is a symmetric complex Hadamard matrix. Additionally, if $M\in\{A,L\}$, then $\sqrt{n}U(\tau)$ is also non-dephased.
\end{corollary}

\begin{corollary}
\label{Lem:compHad1}
Suppose uniform mixing occurs in $X$ at time $\tau$. If all entries of $\sqrt{n}U(\tau)$ are algebraic, then $\sqrt{n}U(\tau)$ is a symmetric Butson Hadamard matrix.
\end{corollary}

Lastly, we show that uniform mixing relative to $A$ and $L$ are equivalent for regular graphs.

\begin{proposition}
Let $M\in\{L,Q\}$ and $X$ be a weighted-regular graph. Local $\epsilon$-uniform mixing (resp., local uniform mixing) occurs at vertex $u$ relative to $A$ if and only if it occurs at vertex $u$ relative to $M$. 
\end{proposition}

\begin{proof}
Since $X$ is weighted $k$-regular, $L=kI-A$. If $M=L$, then $U_L(t)=e^{\ii tk}U_A(-t)$ for all $t\in\R$. Let $\{\tau_m\}\subseteq\R$ such that $\displaystyle\lim_{m\rightarrow\infty}\sqrt{n}U_A(\tau_m)\be_u=\bmu$ for some $1$-uniform vector $\bmu\in\C^n$. Since $k\in\supp_{\be_u}$, Proposition~\ref{proplum}(3) implies that $\{e^{-i\tau_mk}\}$ converges to $e^{-i\tau k}$ for some $\tau\in\R$. Consequently, 
\begin{center}
$\displaystyle\bmu=\lim_{m_\ell\rightarrow\infty}\sqrt{n}U_A(\tau_{m_\ell})\be_u=\lim_{m_\ell\rightarrow\infty}\sqrt{n}e^{-i\tau_{m_{\ell}}k}U_L(-\tau_{m_{\ell}})\be_u=e^{-i\tau k}\lim_{m_\ell\rightarrow\infty}\sqrt{n}U_L(-\tau_{m_{\ell}})\be_u.$
\end{center}
Taking the conjugates of the first and last expressions above yields $\displaystyle
\lim_{m_\ell\rightarrow\infty}\sqrt{n}U_L(\tau_{m_{\ell}})\be_u=e^{-i\tau k}\overline{\bmu}.$ Thus, local $\epsilon$-uniform mixing at vertex $u$ relative $A$ implies that it also occurs at vertex $u$ relative to $L$. A similar argument can be used to prove the converse. The case $M=Q$ follows similarly.
\end{proof}

We refer to the vector $\bmu\in\C^n$ in Corollary~\ref{charlum}(2) as the \textit{target state} of local $\epsilon$-uniform mixing. The main difficulty in the study of local $\epsilon$-uniform mixing and $\epsilon$-uniform mixing is the determination the target state $\bmu$ and the complex Hadamard matrix $\mathcal{H}$ in Proposition~\ref{proplum} and Proposition~\ref{Lem:compHad}, respectively.

\section{Necessary conditions}\label{sec:neccon}

\begin{proposition}
\label{prop:conn}
If $X$ admits local $\epsilon$-uniform mixing, then $X$ is a connected graph.
\end{proposition}

\begin{proof}
We prove the contrapositive. Suppose $X$ is disconnected. In this case, $M(X)$ has a block diagonal form, and consequently, $U(t)$ also has a block diagonal form for all $t$. This implies that for each vertex $u$, $U(t)\be_u$ has a zero entry for all $t$. Thus, local $\epsilon$-uniform mixing does occur at each vertex in $X$.
\end{proof}

In what follows, we let $m_\lambda$ denote the multiplicity of $\lambda\in\spec(M)$.

\begin{lemma}
\label{Lem:no1col}
With the assumption in Proposition~\ref{proplum}, the following hold. 
\begin{enumerate}
\item For each $\lambda\in\supp_{\be_u}$, $\sqrt{n}\left\|E_\lambda\be_u\right\|=\left\|E_\lambda\bmu\right\|$. Consequently, $\supp_{\be_u}=\supp_{\bmu}$.
\item 
For each $\lambda\in\spec(M)$, 
\begin{equation}
\label{111}
n(E_\lambda)_{u,u}=m_\lambda+2\sum_{j>\ell}\cos(\theta_j-\theta_{\ell})(E_\lambda)_{j,{\ell}},
\end{equation}
\item For all integers $r\geq 0$,
\begin{equation}
\label{cosp}
n(M^r)_{u,u}=\sum_{j\in V(X)}(M^r)_{j,j}+2\sum_{j>\ell}\cos(\theta_j-\theta_{\ell})(M^r)_{j,{\ell}}.
\end{equation}
\end{enumerate}
\end{lemma}

\begin{proof}
From Proposition~\ref{proplum}(3),
$\displaystyle\lim_{m\rightarrow\infty}\sqrt{n}e^{\ii \tau_m\lambda}E_\lambda\be_u=E_\lambda\bmu$ for all $\lambda\in\supp_{\be_u}$. Taking modulus yields 1. To prove 2, we again use Proposition~\ref{proplum}(3) which states that for each eigenvector $\bv$ associated with $\lambda\in\supp_{\be_u}$, we have $\displaystyle\lim_{m\rightarrow\infty}\sqrt{n}e^{\ii \tau_m\lambda}(\bv^T\be_u)=\bv^T\bmu$. From this, we get
\begin{center}
$n(\bv\bv^T)_{u,u}=\big(\overline{\displaystyle\lim_{m\rightarrow\infty}\sqrt{n}e^{\ii \tau_m\lambda}\bv^T\be_u}\big)^T\big(\displaystyle\lim_{m\rightarrow\infty}\sqrt{n}e^{\ii \tau_m\lambda}\bv^T\be_u\big)=\big(\overline{\bv^T\bmu}\big)^T(\bv^T\bmu)=\overline{\bmu}^T(\bv\bv^T)\bmu$.
\end{center}
Since $\bmu=\sum_{j}e^{\ii \tau\theta_j}\be_j$, simplifying this equation yields $n(\bv\bv^T)_{u,u}=\|\bv\|^2+2\sum_{j>\ell}\cos(\theta_j-\theta_\ell)(\bv\bv^T)_{j,\ell}$. Summing this equation over an orthonormal basis of eigenvectors associated with $\lambda$ gives us (\ref{111}). Finally, we prove 3. Since $m_\lambda$ is equal to the trace of $E_\lambda$, we may write (\ref{111}) as
\begin{center}
$n(E_\lambda)_{u,u}=\sum_{j}(E_\lambda)_{j,j}+2\sum_{j>\ell}\cos(\theta_j-\theta_{\ell})(E_\lambda)_{j,{\ell}}.$
\end{center}
Multiplying both sides of the above equation by $\lambda^r$ and summing over all $\lambda\in\spec(M)$ yields 3.
\end{proof}

Let $X$ be an unweighted graph.
We denote the average degree, maximum degree, and the number of common neighbors of vertices $j$ and $\ell$ in $X$ by $\overline{d}$, $\Delta$, and $c_{j,\ell}$, respectively.

\begin{corollary}
\label{Cor:eqs}
Let $X$ be an unweighted graph on $n$ vertices. If local $\epsilon$-uniform mixing occurs at $u$ with target state $\bmu=[e^{\ii \theta_1},e^{\ii \theta_2},\ldots,e^{\ii \theta_{n}}]^T$, then the following hold.
\begin{enumerate}
\item If $M=A$, then
\begin{equation}
\label{Eq:cos}
\sum_{\{j,\ell\}\in E(X)}\cos(\theta_j-\theta_{\ell})=0,
\end{equation}
and
\begin{equation}
\label{Eq:cos1}
|E(X)|+\sum_{j>\ell}\cos(\theta_j-\theta_{\ell})c_{j,\ell}=\frac{1}{2}n\operatorname{deg}u.
\end{equation}
\item If $M=Q$, then
\begin{equation}
\label{Eq:cosQ}
|E(X)|+\sum_{\{j,\ell\}\in E(X)}\cos(\theta_j-\theta_{\ell})=\frac{1}{2}n\operatorname{deg}u.
\end{equation}
\item If $M=L$, then
\begin{equation}
\label{Eq:cosL}
|E(X)|-\sum_{\{j,\ell\}\in E(X)}\cos(\theta_j-\theta_{\ell})=\frac{1}{2}n\operatorname{deg}u.
\end{equation}
\end{enumerate}
\end{corollary}

\begin{proof}
Taking $M=A$ and $r\in\{1,2\}$ in Lemma~\ref{Lem:no1col}(3) yields the two equations in (1), while taking $M=Q$ (respectively, $M=L$) and $r=1$ in Lemma~\ref{Lem:no1col}(3) yields (2) (respectively, (3)).   
\end{proof}

The next two results rule out vectors from being target states of local $\epsilon$-uniform mixing.

\begin{corollary}
\label{no1}
Let $X$ be an unweighted graph on $n$ vertices that admits local $\epsilon$-uniform mixing at vertex $u$. If $M\in\{A,L\}$, then the entries of $\bmu$ cannot be all equal.
\end{corollary}

\begin{proof}
If $\bmu=e^{\ii\theta}\one$, then $\sum_{\{j,\ell\}\in E(X)}\cos(\theta_j-\theta_{\ell})=|E(X)|$, a contradiction to Corollary~\ref{Cor:eqs}(1,3).
\end{proof}

\begin{corollary}
\label{edgexp1}
Let $M=A$ and $X$ be an unweighted graph on $n$ vertices. If one of the conditions below hold, then $\bmu=[e^{\ii \theta_1},e^{\ii \theta_2},\ldots,e^{\ii \theta_{n}}]^T$ is not a target state for local $\epsilon$-uniform mixing in $X$.
\begin{enumerate}
\item There is a partition $\bigcup_{j=1}^k V_j$ of $V(X)$ such that $\bmu$ is constant on each $V_j$ and $\sum_{j=1}^k|E(V_j)|>
\sum_{1\leq \ell<j}|E(V_j,V_\ell)|$.
\item All the $\theta_j$'s (mod $2\pi$) belong to an interval $[a,b]$, where $b-a\leq \frac{\pi}{2}$.
\end{enumerate}
\end{corollary}

\begin{proof}
If $\bmu$ is a target state for local $\epsilon$-uniform mixing in $X$, then Corollary~\ref{Cor:eqs}(1) holds. If condition 1 is true, then $\bmu$ is constant on each $V_j$, and so $\sum_{\{j,\ell\}\in E(V_j)}\cos(\theta_j-\theta_{\ell})=|E(V_j)|$. Making use of (\ref{Eq:cos}), we obtain $0=\sum_{\{j,\ell\}\in E(X)}\cos(\theta_j-\theta_{\ell})\geq  \sum_{j=1}^k|E(V_j)|-
\sum_{1\leq \ell<j}|E(V_j,V_\ell)|$,
a contradiction. On the other hand, if 2 is true, then $\cos(\theta_j-\theta_{\ell})\geq 0$. If all entries of $\bmu$ are equal, then we are done by Corollary~\ref{no1}. However, if $\bmu$ has at least two distinct entries, then $\cos(\theta_j-\theta_{\ell})>0$ for some edge $\{j,\ell\}$. This implies that $\sum_{\{j,\ell\}\in E(X)}\cos(\theta_j-\theta_{\ell})>0$, a contradiction to Corollary~\ref{Cor:eqs}(1).
\end{proof}

We now use the average degree to bound the degree of a vertex admitting local $\epsilon$-uniform mixing.

\begin{theorem}
\label{Cor:avedeg}
Let $M\in\{L,Q\}$ and $X$ be an unweighted graph. If local $\epsilon$-uniform mixing occurs at $u$ in $X$, then $\operatorname{deg}u\leq 2\overline{d}$. In particular, if $\epsilon$-uniform mixing occurs in $X$, then $\Delta\leq 2\overline{d}$.
\end{theorem}

\begin{proof}
Since $|\cos(\theta_j-\theta_{\ell})|\leq 1$, Corollary~\ref{Cor:eqs}(2-3) yields $\operatorname{deg}u\leq \frac{4|E(X)|}{n}=2\overline{d}$.
\end{proof}

We also establish a connection between local $\epsilon$-uniform mixing and the eigenvectors of $M$. 

\begin{theorem}
\label{lum}
Let $X$ be a weighted graph on $n$ vertices. If local $\epsilon$-uniform mixing occurs at $u$, then each eigenvector $\mathbf{v}=[v_1,\ldots,v_n]$ of $M$ satisfies
\begin{center}
$\displaystyle\sqrt{n}|v_u|\leq\sum_{j}\big|v_j\big|.$
\end{center}
In particular, if $\epsilon$-uniform mixing occurs in $X$, then the above equation holds for each vertex $u$ of $X$.
\end{theorem}

\begin{proof}
Suppose $\bmu=\sum_{j}e^{\ii \theta_j}$ is the target state. By Proposition~\ref{proplum}(3), $\displaystyle\lim_{m\rightarrow\infty}\sqrt{n}e^{\ii\tau_m\lambda}(\bv^T\be_u)=(\bv^T\bmu)$ for each eigenvector $\bv$ associated with $\lambda\in\supp_{\be_u}$. Taking modulus and applying triangle inequality gives us $\sqrt{n}\left|v_u\right|=\sqrt{n}\left|\bv^T\be_u\right|=\big|\bv^T\bmu\big|=\big|\sum_{j}e^{\ii \theta_j}\bv^T\be_j\big|\leq \sum_{j}\big|v_j\big|$.
\end{proof}

Given a vector $\bv$ with entries in the set $\{c_1,-c_2,0\}$ with $c_1,c_2>0$, Alencar, de Lima, and Nikiforov characterized graphs for which $\bv$ is an eigenvector for $M\in\{A,L,Q\}$ \cite{alencar2023graphs}. For these graphs, we provide a necessary condition for local uniform mixing. 

\begin{corollary}
\label{Lima}
Let $X$ be a weighted graph on $n$ vertices. Given constants $c_1,c_2>0$ and a partition $S_{1}$, $S_{2}$, and $S_{3}$ of $V(X)$, let $\bv$ be a vector such that $\bv^T\be_w=c_1$ if $w\in S_1$, $\bv^T\be_w=-c_2$ if $w\in S_2$, and $\bv^T\be_w=0$ otherwise. If $\bv$ is an eigenvector for $M\in\{A,L,Q\}$ and local $\epsilon$-uniform mixing occurs at some vertex $u\in S_1$ (resp., $u\in S_2$) relative to $M$, then $n\leq \big(|S_1|+\frac{c_2}{c_1}|S_2|\big)^2$ (resp., $n\leq \big(\frac{c_1}{c_2}|S_1|+|S_2|\big)^2$).
\end{corollary}

\begin{proof}
If local uniform mixing occurs at $u\in S_1$ (respectively, $u\in S_2$), then applying Theorem~\ref{lum} gives us $\sqrt{n}c_1\leq c_1|S_1|+c_2|S_1|$ (respectively, $\sqrt{n}c_2\leq c_1|S_1|+c_2|S_1|$). From this, the result is immediate.
\end{proof}

Taking $c_1=c_2$ and $u\in S_1\cup S_2$ in Corollary~\ref{Lima} yields the following result.

\begin{corollary}
\label{Cor:01}
Let $X$ be a weighted graph on $n$ vertices. Let $\mathbf{v}$ be an eigenvector for $M$ with entries in $\{-1,0,1\}$ having $r$ nonzero entries. If local $\epsilon$-uniform mixing occurs at $u$ and $\bv^T\be_u\neq 0$, then $n\leq r^2$.
\end{corollary}

\section{Twin subgraphs}
\label{sec:tw}

\begin{definition}
\label{deftw}
Let $G$ and $H$ be induced subgraphs of a graph $X$ with disjoint vertex sets. Let $f:V(G)\rightarrow V(H)$ be a bijection such that $A(X)_{u,w}=A(X)_{f(u),w}$ for all $u\in V(G)$ and $w\in V(X)\setminus (V(G)\cup V(H))$.
\begin{enumerate}
\item If $f$ is an isomorphism between $G$ and $H$, and there are no edges between $V(G)$ and $V(H)$ in $X$, then $G$ and $H$ are called \textit{false twin subgraphs} in $X$. 
\item If $G$ and $H$ are weighted-regular with the same valency and the bipartite subgraph induced by the edges between $G$ and $H$ is weighted-regular, then  $G$ and $H$ are called \textit{true twin subgraphs} in $X$.
\end{enumerate}
\end{definition}

There are two fundamental differences between false twin subgraphs and true twin subgraphs. The first is that false twin subgraphs are required to be isomorphic, whereas true twin subgraphs are not. And second, the bipartite subgraph $X$ induced by the edges between the false twin subgraphs must be empty, whereas they need not be empty for true twin subgraphs. 

\begin{example}
In Figure \ref{fii}, $X$ and $Y$ contain a pair of false twin subgraphs $G$ and $H$ that are isomorphic to $P_3$. As $G$ is not regular, $G$ and $H$ are not true twin subgraphs of $X$ and $Y$. Meanwhile, in Figure \ref{fi}, $X$ and $Y$ contain a pair of true twin subgraphs $G$ and $H$ that are isomorphic to $K_2$. In particular, if we add a perfect matching between $G$ and $H$, then $G$ and $H$ remain true twin subgraphs in the resulting graphs.
\end{example}

For false twin subgraphs, we have the following result that holds for $M=A$.

\begin{lemma}
\label{lem:x-x}
Let $X$ be a connected weighted graph with false twin subgraphs $G$ and $H$. If $\bx$ is an eigenvector of $A(G)$ for $\lambda\in\spec(A(G))$, then $[\bx,-\bx,\zero]^T$ is an eigenvector of $A(X)$ for $\lambda\in\spec(A(X))$.
\end{lemma}

\begin{proof}
Let $Z$ be the subgraph of $X$ induced by $V(X)\backslash (V(G)\cup V(H))$. By definition of false twin subgraphs, we may partition $A(X)$ conformally relative to $V(G)$, $V(H)$ and $V(G)\backslash (V(G)\cup V(H))$ as
\begin{center}
$A(X)=\begin{bmatrix}A(G)&O&Y \\  O&A(G)&Y\\ Y^T&Y^T&A(Z)\end{bmatrix}$
\end{center}
where $Y$ is some rectangular matrix. Thus, $A(X)[\bx,-\bx,\zero]^T=\lambda[\bx,-\bx,\zero]^T$.
\end{proof}

For true twin subgraphs, we have the following result that works for $M\in\{A,L,Q\}$.

\begin{lemma}
\label{1-10}
If $X$ is a weighted graph on $n$ vertices with a pair of true twin subgraphs, each with $a$ vertices, then the vector $[\one_a ,-\one_a ,\zero]^T$ is an eigenvector for $M$.
\end{lemma}

\begin{proof}
Consider $Z$ as in Lemma \ref{lem:x-x}. By definition of true twin subgraphs, we may partition $A(X)$ and $D(X)$ conformally relative to $V(G)$, $V(H)$ and $V(G)\backslash (V(G)\cup V(H))$ as
\begin{center}
$A(X)=\begin{bmatrix}A(G)&B&Y \\  B&A(H)&Y\\ Y^T&Y^T&A(Z)\end{bmatrix}\quad $ and $\quad D(X)=\begin{bmatrix}D_1&O&O \\  O&D_2&O\\ O&O&D_3\end{bmatrix}$.
\end{center}
Let $\bv=[\one_a ,-\one_a ,\zero]^T$. Since $A(G)\one=A(H)\one=k\one$, a simple calculation reveals that 
\begin{center}
 $A(X)\bv=[(A(G)-B)\one_a,(B-A(H))\one_a,\zero]^T=(k-\ell)\bv$.
\end{center}
Finally, if $p$ is the degree of vertex $u$ in the bipartite graph induced by the edges between $G$ and $H$, then $D(X)\bv=(k+\ell+p)\bv$, and so $\bv$ is an eigenvector for $M\in\{A,L,Q\}$.
\end{proof}

We now prove one of our main results in this section, which applies to graphs with false twin subgraphs.

\begin{theorem}
\label{regtwinsubg}
Let $X$ be a weighted graph on $n$ vertices with false twin subgraphs $G$ and $H$, each with $a$ vertices. If local $\epsilon$-uniform mixing occurs at vertex $u\in V(G)\cup V(H)$ in $X$ relative to $A$, then for each eigenvector $\mathbf{v}=[v_1,\ldots,v_a]$ of $A(G)$,
\begin{center}
$\displaystyle\sqrt{n}|v_u|\leq 2\sum_{j}|v_j|\leq 2a\rho$
\end{center}
where $\rho=\max_{j}|v_j|$. In particular, if $\epsilon$-uniform mixing occurs in $X$ relative to $A$, then $n\leq 4a^2$.
\end{theorem}

\begin{proof}
The first statement follows from Lemma~\ref{lem:x-x} and the argument used in Theorem~\ref{lum}. If $\epsilon$-uniform mixing occurs in $X$, then the above equation holds for each $u\in V(G)$. Let $\mathbf{v}$ be an eigenvector of $A(G)$, and $w$ be an index where $\rho=\max_j|v_j|$ occurs. Then $|v_w|=\rho$, and so $\sqrt{n}\leq 2a.$
\end{proof}

The following result follows from the first statement in Theorem~\ref{regtwinsubg}.

\begin{corollary}
\label{cor:pq}
Let $X$ be a weighted graph on $n$ vertices with false twin subgraphs $G$ and $H$ such that local $\epsilon$-uniform mixing occurs at vertex $u\in V(G)\cup V(H)$ relative to $A$. If $G$ is a weighted graph on $a$ vertices with an eigenvector $\bv=[\one_p,-\one_q,\zero_{a-p-q}]^T$ such that $\bv^T\be_u\neq 0$, then $n\leq 4(p+q)^2.$
\end{corollary}

\begin{example}
\label{ex1}
In Figure \ref{fii}, $G$ is isomorphic to $P_3$, and $G$ and $H$ are false twin subgraphs of $X$ and $Y$. Given the labelling of $G$ in Figure \ref{fii}, we get that $\bv=[1,-1,0]^T$ is an eigenvector for $A(P_3)$. Applying Corollary~\ref{cor:pq} with $p=q=1$, we get that local $\epsilon$-uniform mixing does not occur at vertex $u\in\{1,2,7,6\}$ in $X$ (resp., $Y$) whenever $|V(Z)|>10$ (resp, $|V(Z_1)|+|V(Z_2)|>10$).
\end{example}

Next, we prove the other main result in this section that applies to graphs with true twin subgraphs.
\begin{figure}
\begin{multicols}{2}
\begin{center}
\begin{tikzpicture}[scale=.5,auto=left]
                       \tikzstyle{every node}=[circle, thick, fill=white, scale=0.6]
                       
		        \node[draw] (1) at (1.2,0) {$4$};		   
                \node[draw,minimum size=0.7cm, inner sep=0 pt](6) at (-4.9, 1.5) {$1$};
		        \node[draw,minimum size=0.7cm, inner sep=0 pt] (2) at (-3, 1.5) {$3$};
		        \node[draw,minimum size=0.7cm, inner sep=0 pt] (3) at (-1.1, 1.5) {$2$};
		        \node[draw,minimum size=0.7cm, inner sep=0 pt] (4) at (-3, -1.5) {$5$};
		        \node[draw,minimum size=0.7cm, inner sep=0 pt] (5) at (-1.1, -1.5) {$6$};
                \node[draw,minimum size=0.7cm, inner sep=0 pt] (7) at (-4.9, -1.5) {$7$};

    \node at (-3, 2.4) {$G$};
				\node at (-3,-2.4) {$H$};
				\node at (2.7,-1.5) {$Z$};
				
				\draw[dotted] (-3,1.5) ellipse (3 cm and 1.3 cm);
				\draw[dotted] (-3,-1.5) ellipse (3 cm and 1.3 cm);
				\draw[dotted] (2.7,0.1) circle (2.25 cm);
								
				\draw [thick, black!70] (1)--(3)--(2)--(6);
				\draw [thick, black!70] (5)--(4)--(7);
                \draw [thick, black!70] (1)--(5);

				\draw[thick, black!70] (1)..controls (3,4) and (7,-1)..(1);

				\end{tikzpicture}
                
\begin{tikzpicture}[scale=.5,auto=left]
                       \tikzstyle{every node}=[circle, thick, fill=white, scale=0.6]
                       
		        \node[draw] (1) at (3.1,0) {$4$};
                \node[draw] (6) at (-5.2,0) {$8$};
		        \node[draw,minimum size=0.7cm, inner sep=0 pt] (2) at (-3, 1.5) {$1$};
		        \node[draw,minimum size=0.7cm, inner sep=0 pt] (3) at (-1.1, 1.5) {$3$};
                \node[draw,minimum size=0.7cm, inner sep=0 pt] (7) at (0.8, 1.5) {$2$};
		        \node[draw,minimum size=0.7cm, inner sep=0 pt] (4) at (-3, -1.5) {$7$};
		        \node[draw,minimum size=0.7cm, inner sep=0 pt] (5) at (-1.1, -1.5) {$5$};
                \node[draw,minimum size=0.7cm, inner sep=0 pt] (8) at (0.8, -1.5) {$6$};
	
				\node at (-1.1, 2.4) {$G$};
				\node at (-1.1,-2.4) {$H$};
                \node at (4.5,-1.5) {$Z_1$};
                \node at (-6.7,-1.5) {$Z_2$};
				
				\draw[dotted] (-1.1,1.5) ellipse (3 cm and 1.3 cm);
				\draw[dotted] (-1.1,-1.5) ellipse (3 cm and 1.3 cm);
				\draw[dotted] (4.5,0.1) circle (2.25 cm);
                \draw[dotted] (-6.7,0.1) circle (2.25 cm);
								
				\draw [thick, black!70] (1)--(8)--(5);
                \draw [thick, black!70] (2)--(6)--(4);
				\draw [thick, black!70] (5)--(4);
                \draw [thick, black!70] (1)--(7)--(3)--(2);

				\draw[thick, black!70] (1)..controls (4.7,4) and (8.6,-1)..(1);
                \draw[thick, black!70] (6)..controls (-6.8,4) and (-10.8,-1)..(6);

				\end{tikzpicture}
				
\end{center}
\end{multicols}
\caption{\label{fii} Graphs $X$ (left) and $Y$ (right) with false twin subgraphs $G$ and $H$ isomorphic to $P_3$}
\end{figure}
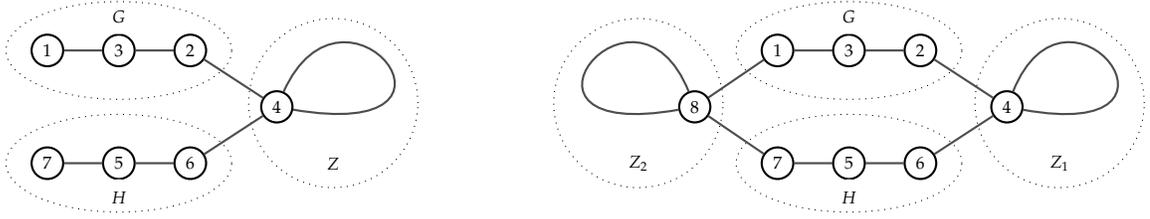

\begin{corollary}
\label{twinsubg}
Let $M\in\{A,L,Q\}$ and $X$ be a weighted graph on $n$ vertices with true twin subgraphs $G$ and $H$, each with $a$ vertices. If local $\epsilon$-uniform mixing at some vertex in $V(G)\cup V(H)$, then $n\leq 4a^2$.
\end{corollary}

\begin{proof}
If $X$ contains a pair of true twin subgraphs, each with $a$ vertices, then by Lemma~\ref{1-10}, $[\one_a ,-\one_a ,\zero]^T$ is an eigenvector for $M\in\{A,L,Q\}$. Now, because $X$ admits local $\epsilon$-uniform mixing at some vertex, applying Corollary~\ref{Cor:01} with $r=2a$ yields the desired result.    
\end{proof}

\begin{figure}
\begin{multicols}{2}
\begin{center}
\begin{tikzpicture}[scale=.5,auto=left]
                       \tikzstyle{every node}=[circle, thick, fill=white, scale=0.6]
                       
		        \node[draw] (1) at (1.2,0) {$3$};		        
		        \node[draw,minimum size=0.7cm, inner sep=0 pt] (2) at (-3, 1.5) {$1$};
		        \node[draw,minimum size=0.7cm, inner sep=0 pt] (3) at (-1.1, 1.5) {$2$};
		        \node[draw,minimum size=0.7cm, inner sep=0 pt] (4) at (-3, -1.5) {$5$};
		        \node[draw,minimum size=0.7cm, inner sep=0 pt] (5) at (-1.1, -1.5) {$4$};		       

    \node at (-2.05, 2.1) {$G$};
				\node at (-2.05,-2.1) {$H$};
				\node at (2.7,-1.5) {$Z$};
				
				\draw[dotted] (-2,1.5) ellipse (2 cm and 1.1 cm);
				\draw[dotted] (-2,-1.5) ellipse (2 cm and 1.1 cm);
				\draw[dotted] (2.7,0.1) circle (2.25 cm);
								
				\draw [thick, black!70] (1)--(3)--(2);
				\draw [thick, black!70] (5)--(4);
                \draw [thick, black!70] (1)--(5);
				\draw [thick,dashed, black!70] (2)--(5);
                \draw [thick,dashed, black!70] (3)--(4);

				\draw[thick, black!70] (1)..controls (3,4) and (7,-1)..(1);

				\end{tikzpicture}
                
\begin{tikzpicture}[scale=.5,auto=left]
                       \tikzstyle{every node}=[circle, thick, fill=white, scale=0.6]
                       
		        \node[draw] (1) at (1.2,0) {$3$};
                \node[draw] (6) at (-5.2,0) {$6$};
		        \node[draw,minimum size=0.7cm, inner sep=0 pt] (2) at (-3, 1.5) {$1$};
		        \node[draw,minimum size=0.7cm, inner sep=0 pt] (3) at (-1.1, 1.5) {$2$};
		        \node[draw,minimum size=0.7cm, inner sep=0 pt] (4) at (-3, -1.5) {$5$};
		        \node[draw,minimum size=0.7cm, inner sep=0 pt] (5) at (-1.1, -1.5) {$4$};		       
	
				\node at (-2.05, 2.1) {$G$};
				\node at (-2.05,-2.1) {$H$};
                \node at (2.7,-1.5) {$Z_1$};
                \node at (-6.7,-1.5) {$Z_2$};
				
				\draw[dotted] (-2,1.5) ellipse (2 cm and 1.1 cm);
				\draw[dotted] (-2,-1.5) ellipse (2 cm and 1.1 cm);
				\draw[dotted] (2.7,0.1) circle (2.25 cm);
                \draw[dotted] (-6.7,0.1) circle (2.25 cm);
								
				\draw [thick, black!70] (1)--(3)--(2);
                \draw [thick, black!70] (2)--(6)--(4);
				\draw [thick, black!70] (5)--(4);
                \draw [thick, black!70] (1)--(5);
                \draw [thick,dashed, black!70] (2)--(4);
                \draw [thick,dashed, black!70] (3)--(5);

				\draw[thick, black!70] (1)..controls (3,4) and (7,-1)..(1);
                \draw[thick, black!70] (6)..controls (-6.8,4) and (-10.8,-1)..(6);

				\end{tikzpicture}
				
\end{center}
\end{multicols}
\caption{\label{fi} Graphs with true twin subgraphs $G$ and $H$ isomorphic to $K_2$; $G$ and $H$ remain true twin subgraphs after inserting the dashed edges}
\end{figure}
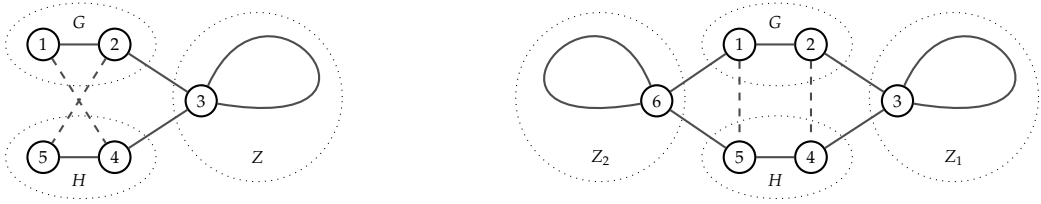

The next result is obtained from combining Theorem~\ref{regtwinsubg} and Corollary~\ref{twinsubg}.

\begin{corollary}
\label{corintro}
Let $M\in\{A,L,Q\}$ and $X$ be a weighted graph on $n$ vertices with (false or twin) twin subgraphs, each with $a$ vertices. If $\epsilon$-uniform mixing occurs in $X$, then $n\leq 4a^2$.
\end{corollary}

\begin{example}
\label{ex2}
In Figure \ref{fi}, $G$ is isomorphic to $K_2$, and $G$ and $H$ are true twin subgraphs of $X$ and $Y$. Applying Corollary~\ref{twinsubg} with $a=2$, the following hold relative to $M\in\{A,L,Q\}$.
\begin{enumerate}
\item If $|V(Z)|>12$, then local $\epsilon$-uniform mixing does not occur at $u\in V(G)\cup V(H)$ in $X$. 
\item If $|V(Z_1)|+|V(Z_2)|>12$, then local $\epsilon$-uniform mixing does not occur at $u\in V(G)\cup V(H)$ in $Y$.
\end{enumerate}
Additionally, if a perfect matching is added between $G$ and $H$, then the two statements above remain true.
\end{example}

If we insert edges between $G$ and $H$, then Example~\ref{ex2}(1-2) remain true as long as the bipartite subgraph induced by these edges is regular. In this case, $G$ and $H$ continue to be true twin subgraphs of the resulting graph, and so Corollary~\ref{twinsubg} still applies. However, the addition of edges between $G$ and $H$ in Example~\ref{ex1} destroys the property that $G$ and $H$ are false twin subgraphs, so we do not know whether Example~\ref{ex1}(1-2) still hold since Corollary~\ref{cor:pq} no longer applies to the resulting graph.

Let $N(u)$ denote the neighborhood of vertex $u$ in $X$. Vertices $u$ and $v$ are \textit{twins} in $X$ if $N(u)\backslash\{v\}=N(v)\backslash\{u\}$ and the edges $\{u,w\}$ and $\{v,w\}$ have equal weights for all $w\notin V(X)\backslash\{u,v\}$ \cite{Kirkland2023}. In particular, adjacent twins and non-adjacent twins are \textit{true twins} and \textit{false twins}, respectively. Note that true and false twins are singleton true and false twin subgraphs, respectively. Thus, our definitions of true and false twin subgraphs generalize the notion of true and false twins. Moreover, we may also view false twins as singleton true twin subgraphs, where the subgraph induced by each vertex and the bipartite subgraph between them are both 0-regular. Applying Corollary~\ref{twinsubg} with $a=1$ yields the next result.

\begin{corollary}
\label{twin}
Let $X$ be a weighted graph on $n$ vertices containing a vertex $u$ with a twin. If local $\epsilon$-uniform mixing occurs at $u$ in $X$, then $n\leq 4$.
\end{corollary}

\section{Periodicity}\label{sec:per}
 
We say that $\supp_{\bx}$ satisfies the \textit{ratio condition} if $\frac{\alpha-\beta}{\gamma-\zeta}\in\mathbb{Q}$ for all $\alpha,\beta,\gamma,\zeta\in \supp_{\bx}$ with $\gamma\neq\zeta$. The result below follows from \cite[Theorem 3.2]{godsil2025perfect}.

\begin{theorem}
\label{thm:ratiocon}
The vector $\bx\in\C^n$ is periodic if and only if $\supp_{\bx}$ satisfies the ratio condition. Moreover, if $M$ has integer entries, $\bx$ is rational and $|\supp_{\bx}|\geq 3$, then $\bx$ is periodic if and only if either (i) $\supp_{\bx}\subseteq\Z$, or (ii) each $\lambda_j\in \supp_{\bx}$ is of the form $\lambda_j=\frac{1}{2}(a+b_j\sqrt{\Delta})$, where $a,b_j,\Delta$ are integers and $\Delta>1$ is square-free.
\end{theorem}

The following result can be established using the same argument used in 
\cite[Theorem 3.4]{Godsil2013}.

\begin{theorem}
\label{algent}
Let $M$ be an integer matrix and $\bx\in\C^n$ be a vector whose entries are all algebraic numbers. If the entries of $U(\tau)\bx$ are all algebraic numbers for some $\tau>0$, then the ratio of any two nonzero eigenvalues in $\supp_{\bx}$ is rational and $\bx$ is periodic. 
\end{theorem}

\begin{corollary}
\label{cor:lumper}
Let $M$ be an integer matrix and $\bx$ be a vector with all rational entries. If all entries of $U(\tau)\bx$ are algebraic numbers for some $\tau>0$, then $\bx$ is periodic, and 
the following hold whenever $|\supp_{\bx}|\geq 3$.
\begin{enumerate}
\item If $M\in\{L,Q\}$, then $\supp_{\bx}\subseteq\Z$. 
\item If $M=A$, then either (i) $\supp_{\bx}\subseteq\Z$, or (ii) $\supp_{\bx}$ is closed under multiplication of $-1$ and each $\lambda_j\in\supp_{\bx}$ is of the form $\pm b_j\sqrt{\Delta}$, where $b_j$ and $\Delta$ are integers and $\Delta>1$ is square-free. Additionally, if $\bx$ is not orthogonal to the Perron eigenvector of $A$, then $X$ is bipartite in (ii).
\end{enumerate} 
\end{corollary}

\begin{proof}
Let $|\supp_{\bx}|\geq 3$.
Since $M$ has integer entries and $\bx$ is rational, Theorem~\ref{thm:ratiocon} and Theorem~\ref{algent} together imply that (i) $\supp_{\bx}\subseteq\Z$, or (ii) each $\lambda_j\in \supp_{\bx}$ is of the form $\lambda_j=\frac{1}{2}(a+b_j\sqrt{\Delta})$, where $a,b_j,\Delta$ are integers and $\Delta>1$ is square-free. Now, suppose $M\in\{L,Q\}$. Using the fact that the ratio of any two nonzero eigenvalues in $\supp_{\bx}$ is rational and $\supp_{\bx}$ is closed under taking algebraic conjugates, (ii) implies that $\frac{a+b_j\sqrt{\Delta}}{a-b_j\sqrt{\Delta}}\in \Q$. This is only possible when $a=0$, in which case $\pm b_j\sqrt{\Delta}$ are eigenvalues of $M$. Since $M$ in this case is positive semidefinite, all its eigenvalues are nonnegative, a contradiction. Thus (ii) cannot happen, and we get 1. For $M=A$, the same argument implies that (i) $\supp_{\bx}\subseteq\Z$, or (ii) each $\lambda_j\in \supp_{\bx}$ is of the form $\lambda_j=\pm b_j\sqrt{\Delta}$. In particular, if (ii) holds and $\bx$ is not orthogonal to the Perron eigenvector of $A$, then the Perron eigenvalue and its negative belong to $\supp_{\bx}$, which implies that $X$ is bipartite.
\end{proof}

By virtue of Corollary~\ref{charlum}(2), taking $\bx=\be_u$ in Theorem~\ref{algent} and Corollary~\ref{cor:lumper} yields the following result.

\begin{corollary}
\label{algent1}
Let $M$ be an integer matrix. If local uniform mixing occurs at $u$ with a target state $\bmu$ that has all algebraic entries, then $u$ is periodic and conditions 1 and 2 in Corollary~\ref{cor:lumper} hold with $\bx=\be_u$.
\end{corollary}

\begin{lemma}
\label{lem:sup}
If $u\in V(X)$, then $\supp_{u}\cap \supp_{v}$ contains a nonzero eigenvalue for some vertex $v\neq u$.
\end{lemma}

\begin{proof}
Let $u\in V(X)$. If $\supp_{u}\cap \supp_{w}=\varnothing$ for all $w\in V(X)\backslash\{v_1\}$, then $M$ is permutation similar to a block-diagonal matrix, a contradiction to the irreducibility of $M$. So there exists $v\in V(X)\backslash \{u\}$ such that $S:=\supp_{u}\cap \supp_{v}\neq \varnothing$. If $0E_0$ exists as a term in the spectral decomposition of $A$, then it does not contribute to $A$ entrywise. Thus, it must be that $S\neq \{0\}$. Otherwise, $M$ is again permutation similar to a block-diagonal matrix, a contradiction. Thus, there is some $\lambda\in S$ such that $\lambda\neq 0$.
\end{proof}

Using Lemma~\ref{lem:sup}, we reprove the following result that first appeared in \cite[Theorem 3.4]{Godsil2013}.

\begin{corollary}
Let $M$ be an integer matrix. If all entries of $U(\tau)$ at some $\tau>0$ are algebraic numbers, then the ratio of any two nonzero eigenvalues in $\spec(M)$ is rational.
\end{corollary}

\begin{proof}
Our assumption and Theorem~\ref{algent} imply that the ratio of any two nonzero eigenvalues in $\supp_{\be_u}$ is rational for all $u\in V(X)$. Fix $v_1\in V(X)$. By Lemma~\ref{lem:sup}, there exists $v_2\in V(X)\backslash \{v_1\}$ such that $S_1:=\supp_{v_1}\cap \supp_{v_2}$ contains an eigenvalue $\lambda_1\neq 0$. Thus, each $\theta\in T_1:=\supp_{v_1}\cup\supp_{v_2}$ can be expressed as $\theta=r_1\lambda_1$, for some $r_1\in\mathbb{Q}$. Applying the same argument, we get another vertex $v_3\neq v_1,v_2$ such that $S_2:=T_1\cap\supp_{v_3}\neq \varnothing$. If $\lambda_2\in S_2$, then each $\theta\in T_2:=T_1\cup\supp_{v_3}$ can be written as $\theta=r_2\lambda_2$ for some $r_2\in\mathbb{Q}$. But as $\lambda_2=r_1\lambda_1$, we get $\theta=r_2r_1\lambda_1$ for some $r_2,r_1\in\mathbb{Q}.$ Arguing inductively, we get that each $\theta\in T_{n-1}=\bigcup_{u\in V(X)}\supp_{\be_u}=\spec(M)$ can be written as $\theta=r_{n-1}\cdots r_1\lambda_1$ for some $r_1,\ldots,r_{n-1}\in\mathbb{Q}$. Therefore, ratio of any two nonzero eigenvalues in $\spec(M)$ is rational.
\end{proof}

Adapting the proof of Corollary~\ref{cor:lumper} yields the next result.
\begin{corollary}
\label{cor:lumper1}
Let $M$ be an integer matrix. If all entries of $U(\tau)$ at some $\tau>0$ are algebraic numbers, then $X$ is periodic, and and 
the following statements hold.
\begin{enumerate}
\item If $M\in\{L,Q\}$, then $\spec(M)\subseteq\Z$. 
\item If $M=A$, then either (i) $\spec(M)\subseteq\Z$, or (ii) $X$ is bipartite, and each $\lambda_j\in\spec(M)$ is of the form $\pm b_j\sqrt{\Delta}$, where $b_j$ and $\Delta$ are integers and $\Delta>1$ is square-free.
\end{enumerate} 
\end{corollary}

Combining Corollaries \ref{Lem:compHad1} and \ref{cor:lumper1} yields the following result.

\begin{corollary}
\label{algent2}
Let $M$ be an integer matrix. If $X$ admits uniform mixing and $\sqrt{n}U(\tau)$ is a Butson Hadamard matrix, then $X$ is periodic, and conditions 1 and 2 in Corollary~\ref{cor:lumper1} hold. 
\end{corollary}

The following result complements Corollary~\ref{algent1}, and is immediate from \cite[Lemma 5.1]{godsil2025perfect}.

\begin{proposition}
\label{per2t}
Suppose local uniform mixing occurs at $u$ in $X$ at time $\tau$. If the target state $\bmu$ is equal to $\alpha\bz$ for some $|\alpha|=1$ and a real vector $\bz$, then vertex $u$ is periodic at time $2\tau$.
\end{proposition}

It is known that local uniform mixing occurs at a vertex of $K_4$ at $\tau=\frac{\pi}{4}$ where $\bmu=ie^{\ii\tau}[1,-1,-1,-1]^T$. Invoking Proposition~\ref{per2t}, we get the well-known fact that each vertex of $K_4$ is periodic at $2\tau=\frac{\pi}{2}$. However, if $\bmu$ in Corollary~\ref{algent1} is not a unit scalar multiple of some real vector, then the time at which vertex $u$ is periodic is not necessarily a rational multiple of the time $\tau$ at which local uniform mixing occurs. For example, if we take $X=P_3$ and $M=A$, then one checks that the degree two vertex $u$ of $P_3$ is periodic at time $\tau'=\frac{\pi}{\sqrt{2}}$ and admits local uniform mixing at $\tau=\frac{\arctan\sqrt{2}}{\sqrt{2}}$ with target state $\bmu=[i,-1,i]^T$. Since $\arctan\sqrt{2}$ is not a rational multiple of $\pi$, it follows that $\tau'$ is not a rational multiple of $\tau$.

\begin{theorem}
\label{raregen}
Let $M\in\{A,L,Q\}$. For each integer $k>0$, there are only finitely many connected integer-weighted graphs with maximum degree at most $k$ such that for a nonnegative rational vector $\bx$, all entries of $U(\tau)\bx$ are algebraic numbers for some $\tau>0$.
\end{theorem}
\begin{proof}
If $M$ has integer entries and $\bx$ is a nonnegative rational vector such that all entries of $U(\tau)\bx$ are algebraic numbers for some $\tau>0$, then $\bx$ is periodic by Corollary~\ref{cor:lumper}. Invoking Theorem 3.8 and Remark 3.9 in \cite{godsil2025perfect} yields the desired conclusion.
\end{proof}

\section{Bipartite graphs}\label{sec:bip}

Throughout, we assume $M=A$. Let $X$ be a weighted bipartite graph with parts $B_1$ and $B_2$. We may write
\begin{center}
 $A=\left[\begin{array}{ccc} O&B \\ B^T&O \end{array} \right]$
\end{center}
for some $|B_1|\times |B_2|$ matrix $B$. It is known that $\lambda$ is an eigenvalue of $A$ with eigenvector $[\bu,\bv]^T$ if and only if $-\lambda$ an eigenvalue of $A$ with eigenvector $[\bu,-\bv]^T$. Consequently, we may assume that an eigenvector associated with the zero eigenvalue of $A$ has the form $[\bu,\zero]^T$ or $[\zero,\bv]^T$. 

In \cite{Godsil2013}, it was shown that the transition matrix of a weighted bipartite graph has the form
\begin{equation}
\label{Eq:bip}
U_A(t)=\left[\begin{array}{ccc} U_1(t)&iU_2(t) \\ iU_2(t)^T&U_3(t) \end{array} \right]\quad \text{for all $t\in\R$},
\end{equation}
where $U_j(t)$ is a real matrix for each $j\in\{1,2,3\}$. Hence, the following results are immediate.

\begin{proposition}
\label{1}
Let $X$ be a weighted bipartite graph with parts $B_1$ and $B_2$.
\begin{enumerate}
\item If local $\epsilon$-uniform mixing occurs at $u\in B_1$ (resp., $u\in B_2$), then the target state is $[\bmu_1,\ii\bmu_2]^T$ (resp., $[\ii\bmu_1,\bmu_2]^T$), where $\bmu_1$ and $\bmu_2$ are $\pm 1$-vectors indexed by $B_1$ and $B_2$, respectively.
\item If $X$ has uniform mixing in $X$ at $\tau$, then $X$ is periodic and $\sqrt{n}U(\tau)$ is a Turyn Hadamard matrix.
\end{enumerate}
\end{proposition}

The following fact is a restatement of Corollary 9.7.3 in \cite{Coutinho2021}.

\begin{theorem}
Let $X$ be an integer-weighted bipartite graph. The following hold.
\begin{enumerate}
\item If vertex $u$ admits local uniform mixing in $X$, then vertex $u$ is periodic and either
(i) $\supp_{\be_u}\subseteq\Z$, or (ii) each $\lambda_j\in\supp_{\be_u}$ is of the form $\pm b_j\sqrt{\Delta}$, where $b_j$ and $\Delta$ are integers and $\Delta>1$ is square-free.
\item If $X$ admits uniform mixing in $X$ at time $\tau$, then $X$ is periodic, and either (i) $\spec(M)\subseteq\Z$, or (ii) each $\lambda_j\in\spec(M)$ is of the form $\pm b_j\sqrt{\Delta}$, where $b_j$ and $\Delta$ are integers and $\Delta>1$ is square-free.
\end{enumerate}
\end{theorem}


\begin{theorem}
\label{raregenA}
For each integer $k>0$, there are only finitely many connected integer-weighted bipartite graphs with maximum degree at most $k$ that admit local uniform mixing at a vertex relative to $A$.
\end{theorem}

\begin{proof}
Taking $\bx=\be_u$ in Theorem~\ref{raregen} and using the fact that the target state of local uniform mixing in a bipartite graph has all algebraic entries by Proposition~\ref{1}(1) yields the desired conclusion.
\end{proof}

\begin{proposition}
\label{prop7}
Let $X$ be a weighted bipartite graph on $n$ vertices with parts $B_1$ and $B_2$. Suppose local $\epsilon$-uniform mixing occurs at $u\in B_1$ relative to $\{\tau_m\}$ with $\bmu=[\bmu_1,\ii\bmu_2]^T$. Let $\lambda\in\supp_{\be_u}$, $\bv=[\bv_1,\bv_2]^T$ be an eigenvector for $\lambda$ with $\bv^T\be_u\neq 0$, and define 
$\cos(\lambda\tau'):=\displaystyle\lim_{m\rightarrow\infty}\cos(\lambda\tau_m)$ and $\sin(\lambda\tau'):=\displaystyle\lim_{m\rightarrow\infty}\sin(\lambda\tau_m)$.
\begin{enumerate}
\item We have $\cos(\lambda\tau')\sqrt{n}\bv_1^T\be_u=\bv_1^T\bmu_1$ and $\sin(\lambda\tau')\sqrt{n}\bv_1^T\be_u=\bv_2^T\bmu_2.$ 
\item Suppose $\bv$ has integer entries. Then:
\begin{enumerate}
\item We have $n=\frac{a^2+b^2}{c^2}$, where $a,b,c$ are integers and $c=\bv^T\be_u$. In particular, $n$ is a sum of two integer squares whenever $c=\pm 1$ (e.g., when $X$ is regular).
\item $n$ is a perfect square if and only if $\cos(\lambda\tau')$ is rational, if and only if $\sin(\lambda\tau')$ is rational. In particular, if $\cos(\lambda\tau')\in\{-\frac{\sqrt{3}}{2},-\frac{1}{2},\frac{1}{2},\frac{\sqrt{3}}{2}\}$, then $n$ is not a perfect square.
\item Suppose further that all entries of $\bv$ belong to $\{-1,0,1\}$. If $\cos(\lambda\tau')\in\{-1,0,1\}$, then $n$ is a perfect square. In particular, if $\cos(\lambda\tau')=\pm 1$, then $\sqrt{n}\leq |B_1|$ and $\sqrt{n}$ has the same parity as $|B_1|$. Otherwise, $\sqrt{n}\leq |B_2|$ and $\sqrt{n}$ has the same parity as $|B_2|$. 
\end{enumerate}
\end{enumerate}
\end{proposition}

\begin{proof}
Set $\gamma_\lambda=\lim_{m\rightarrow\infty}e^{\ii\tau_m\lambda}$. By assumption, $\bv'=[\bv_1,-\bv_2]^T$ is also an eigenvector for $-\lambda$. Also, $\gamma_\lambda\sqrt{n}[\bv_1,\bv_2]^T\be_u=[\bv_1,\bv_2]^T\bmu$ by Proposition~\ref{proplum}(3). Since $\bmu=[\bmu_1,\ii\bmu_2]^T$, the preceding equation yields
\begin{center}
$\gamma_\lambda\sqrt{n}\bv_1^T\be_u=\bv_1^T\bmu_1+\ii\bv_2^T\bmu_2\neq 0\quad $ and $\quad \gamma_{-\lambda}\sqrt{n}\bv_1^T\be_u=\bv_1^T\bmu_1-\ii\bv_2^T\bmu_2\neq 0$. 
\end{center}
Adding these equations gives us $\cos(\lambda\tau')\sqrt{n}\bv_1^T\be_u=\bv_1^T\bmu_1$ and $\sin(\lambda\tau')\sqrt{n}\bv_1^T\be_u=\bv_2^T\bmu_2.$ So, 1 holds.

To prove 2, suppose $\bv$ has integer entries. As each entry of $\bmu$ is in $\{\pm 1,\pm \ii\}$, the left equation above gives us $\gamma_\lambda\sqrt{n}c=a+b\ii$, where $a,b,c\in\Z$. Taking modulus yields 2(a), while 2b is immediate from 1. To prove 3, suppose all entries of $\bv$ belong to $\{-1,0,1\}$. If $\cos(\lambda\tau')\in\{-1,0,1\}$, then $n$ is a perfect square by 2(b). As $\bv_1^T\be_u\neq 0$, we get $|\bv_1^T\be_u|=1$. By 1, either $\sqrt{n}=|\bv_1^T\bmu_1|\neq 0$ or $\sqrt{n}=|\bv_2^T\bmu_2|\neq 0.$ Since $\bv$ has entries in $\{-1,0,1\}$, both $\bv_1^T\bmu_1$ and $\bv_2^T\bmu_2$ are integers. In particular, if $\cos(\lambda\tau')=\pm 1$, then $\sqrt{n}=|\bv_1^T\bmu_1|$. As $\bv_1$ and $\bmu_1$ have entries in $\{-1,0,1\}$, we get $\bv_1^T\bmu_1\leq |B_1|$, and $\bv_1^T\bmu_1$ and $|B_1|$ have the same parities. So, $\sqrt{n}\leq |B_1|$, and $n$ and $|B_1|$ have the same parities. If $\cos(\lambda\tau')=\pm 0$, then the same argument yields $\sqrt{n}\leq |B_2|$, and $n$ and $|B_2|$ have the same parities.
\end{proof}

The following is immediate from Proposition~\ref{prop7}(2b,2c).

\begin{corollary}
\label{cor:singularbip}
Let $X$ be a singular integer-weighted bipartite graph on $n$ vertices with parts $B_1$ and $B_2$. If local $\epsilon$-uniform mixing occurs at $u\in B_1$, then $n$ is a perfect square. Additionally, if there is an eigenvector $\bv$ for $0$ with entries in $\{-1,0,1\}$ and $\bv^T\be_u\neq 0$, then $\sqrt{n}\leq |B_1|$, and $n$ and $|B_1|$ have the same parities.
\end{corollary}

\begin{theorem}
\label{thm:eumbip}
Suppose $\epsilon$-uniform mixing occurs in a weighted bipartite graph $X$ on $n$ vertices with parts $B_1$ and $B_2$ relative to $\{\tau_m\}$. Consider $\cos(\lambda\tau')$ and $\sin(\lambda\tau')$ as in Proposition~\ref{prop7}. The following hold.
\begin{enumerate}
\item $n\equiv 0$ (mod 4).
\item 
$\min\{|B_1|,|B_2|\}\geq\alpha\sqrt{n}\geq\sqrt{\frac{n}{2}}$, where $\alpha:=\displaystyle\max_{\lambda\in\spec(M)}\{|\cos(\lambda\tau')|,|\sin(\lambda\tau')|\}$. Additionally, if $X$ is singular, then $\min\{|B_1|,|B_2|\}\geq \sqrt{n}$.
\item If $X$ is singular and integer-weighted, then $n$ is an even perfect square. Additionally, if there is an eigenvector for 0 with all entries in $\{-1,0,1\}$, then $|B_1|\equiv |B_2|\equiv d$ (mod 4), where $d\in\{0,2\}$.
\end{enumerate}
\end{theorem}

\begin{proof}
The proof of 1 follows from Proposition~\ref{1}(2) and an argument similar to that of \cite{Godsil2013}. To prove 2, we apply Proposition~\ref{prop7}(1) to each vertex of $X$. First, let $u\in B_1$. Since $|\bv_1^T\bmu_1|\leq |B_1||\bv_1^T\be_w|$, where $w\in B_1$ is an index where $\max_{j\in B_1}|\bv_1^T\be_j|$ occurs, Proposition~\ref{prop7}(1) yields $\sqrt{n}|\cos(\lambda\tau')||\bv_1^T\be_u|\leq |B_1||\bv_1^T\be_w|$. Thus, $|\cos(\lambda\tau')|\sqrt{n}\leq \frac{|B_1|\bv_1^T\be_w}{|\bv_1^T\be_u|}$, and taking $u=w$ yields $|\cos(\lambda\tau')|\sqrt{n}\leq |B_1|$. Next, let $u\in B_2$. Since $|\bv_2^T\bmu_2|\leq |B_1||\bv_2^T\be_w|$, where $w\in B_2$ is an index where $\max_{j\in B_2}|\bv_2^T\be_j|$ occurs, the first equation in Proposition~\ref{prop7}(1) with the interchanged roles of $B_1$ and $B_2$ implies that $\sqrt{n}|\cos(\lambda\tau')||\bv_2^T\be_u|\leq |B_2||\bv_1^T\be_w|$, and so the same argument yields $|\cos(\lambda\tau')|\sqrt{n}\leq |B_2|$. Thus, $|\cos(\lambda\tau')|\sqrt{n}\leq\min\{|B_1|, |B_2|\}$. The same argument applied to the second equation in Proposition~\ref{prop7}(1) yields $|\sin(\lambda\tau')|\sqrt{n}\leq\min\{|B_1|, |B_2|\}$. So, $\min\{|B_1|,|B_2|\}\geq\alpha\sqrt{n}$ with $\alpha\geq\frac{\sqrt{2}}{2}$. In particular, if $X$ is singular and $\lambda=0$, then $\alpha=1$. This proves 2. Finally, 3 follows from Corollary~\ref{cor:singularbip} and the fact that $n$ is even by 1.
\end{proof}

From Theorem~\ref{thm:eumbip}(2), weighted bipartite graphs with $\epsilon$-uniform mixing cannot be `too' unbalanced.

Let $X$ be an unweighted graph. A \textit{subdivision} $S(X)$ of $X$ is obtained from $X$ by replacing each edge $\{u,v\}$ of $X$ by the edges $\{u,w\}$ and $\{w,v\}$. Note that $S(X)$ is always bipartite and has exactly $|V(X)|+|E(X)|$ number of vertices. The following is immediate from Theorem~\ref{thm:eumbip}(1).

\begin{theorem}
\label{thm:subd}
Let $X$ be an unweighted graph. If $|V(X)|+|E(X)|\not\equiv 0$ (mod 4), then $S(X)$ does not admit $\epsilon$-uniform mixing for any assignment of edge weights.
\end{theorem}

Let $k\geq 0$ be an integer. We say that a connected unweighted graph $X$ on $n$ vertices is \textit{$k$-cyclic} if $X$ has $(n-1)+k$ edges. Note that $X$ is a tree if $k=0$ and unicyclic if $k=1$. 
If $X$ is $k$-cyclic, then $S(X)$ has 
$2n-1+k$ number of vertices. A direct application of Theorem~\ref{thm:subd} yields the following result.

\begin{corollary}
\label{cor:subd}
Let $X$ be a $k$-cyclic weighted graph on $n$ vertices. If either $k$ is even, or $k$ and $n+\frac{k-1}{2}$ are both odd, then $S(X)$ does not admit $\epsilon$-uniform mixing for any assignment of edge weights.
\end{corollary}

\begin{corollary}
\label{nodd}
Let $X$ be an unweighted bipartite graph on $n$ vertices admitting local $\epsilon$-uniform mixing at vertex $u$. Then $n$ or $\operatorname{deg}u$ is even. Moreover, there is an even number of vertices in $X$ with degree $d\equiv 2,3$ (mod 4) if and only if $|E(X)|$ and $\frac{1}{2}n\operatorname{deg}u$ have the same parities.
\end{corollary}

\begin{proof}
By Proposition~\ref{1}, the vector $\bmu$ in Lemma~\ref{Lem:no1col} satisfies $\theta_j=\frac{m_j\pi}{2}$ for each $j$, where $m_j$ is some integer. Thus, $\cos(\theta_j-\theta_{\ell})$ is an integer for all $j,\ell\in V(X)$ with $j>\ell$. Invoking (\ref{Eq:cos1}) in Corollary~\ref{Cor:eqs}, we find that $\frac{1}{2}n\operatorname{deg}u$ is an integer. Hence, $n$ or $\operatorname{deg}u$ is even. To complete the proof, we utilize (\ref{Eq:cos1}), which states that
\begin{equation}
\label{eq0}
|E(X)|+\sum_{j>\ell}\cos(\theta_j-\theta_{\ell})c_{j,\ell}=\frac{1}{2}n\operatorname{deg}u.
\end{equation}
If $j$ and $\ell$ belong to different parts, then $c_{j,\ell}=0$ because $X$ is bipartite. On the other hand, if $j$ and $\ell$ belong to the same parts, then $c_{j,\ell}$ may or may not be equal to 0. For the case when $c_{j,\ell}\neq 0$, Proposition~\ref{1}(1) implies that $\cos(\theta_j-\theta_{\ell})=\pm 1$. Thus, $\sum_{j>\ell}\cos(\theta_j-\theta_{\ell})c_{j,\ell}$ has the same parity as $\sum_{j>\ell}c_{j,\ell}$. 
If $v$ is a pendent vertex in $X$, then $v$ is not a common neighbor of any two vertices. Otherwise, $\operatorname{deg}v\geq 2$ and there are ${\operatorname{deg}v\choose 2}$ distinct pairs of vertices having $v$ as a common neighbor. Thus, 
\begin{equation}
\label{q}
\sum_{j>\ell}c_{j,\ell}=\sum_{{v\in V(X)}\atop{\operatorname{deg}v\geq 2}}{\operatorname{deg}v\choose 2},
\end{equation}
where the expression on the right counts the total number of common neighbors of every pair of vertices in $T$. Finally, since $\sum_{j>\ell}\cos(\theta_j-\theta_{\ell})c_{j,\ell}$ and $\sum_{j>\ell}c_{j,\ell}$ have the same parities, and ${\operatorname{deg}v\choose 2}$ is even if and only if $\operatorname{deg}v\equiv 0,1$ (mod 4). Thus, $\sum_{j>\ell}\cos(\theta_j-\theta_{\ell})c_{j,\ell}$ is even if and only if there is an even number of vertices in $X$ with degree $d\equiv 2,3$ (mod 4). Combining this with (\ref{eq0}) yields the desired conclusion.
\end{proof}

\section{Planar graphs}\label{sec:plan}

\begin{figure}
\begin{center}
\begin{tikzpicture}[scale=.5,auto=left]
                       \tikzstyle{every node}=[circle, thick, fill=white, scale=0.75]

                \node[draw,minimum size=0.7cm, inner sep=0 pt] (1) at (1.5, 0) {$v$};
		        \node[draw,minimum size=0.7cm, inner sep=0 pt] (3) at (-1, 1) {$u$};
		        \node[draw,minimum size=0.7cm, inner sep=0 pt] (5) at (-1, -1) {$w$};
				\node at (6,0.5) {$H$};
				\draw[dotted] (3,0.4) circle (2.35 cm);
								
				\draw [thick, black!70] (5)--(1)--(3);

				\draw[thick, black!70] (1)..controls (3,4) and (8,0)..(1);

				\end{tikzpicture}		
\end{center}
\caption{\label{fig1} A graph $X$ with pendent vertices $u$ and $w$ sharing a common neighbor $v$.}
\end{figure}
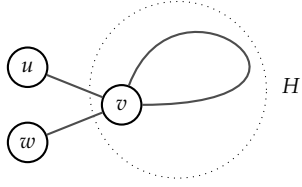

\begin{theorem}
\label{lumplntr}
Let $M\in\{A,L,Q\}$. Almost all connected planar graphs and almost all trees contain a vertex that does not admit local $\epsilon$-uniform mixing for any assignment of edge weights.
\end{theorem}

\begin{proof}
For the case for planar graphs, adapting the proof of \cite[Theorem 8]{godsil2025quantum}, we get that almost all connected planar graphs have the same form as the graph in Figure \ref{fig1}, where $H$ is a connected planar graph. Let $\alpha$ and $\beta$ be the weights of edges $\{u,v\}$ and $\{w,v\}$. One checks that $\bv=\be_u-\frac{\alpha}{\beta}\be_w$ is an eigenvector for $M\in\{A,L,Q\}$. If vertex $u$ does not admit local $\epsilon$-uniform mixing, then we are done. Suppose it does. Our goal is to show that if vertex $w$ also admits local $\epsilon$-uniform mixing, then $n\leq 4$. Assume $w$ admits local $\epsilon$-uniform mixing. Applying Theorem~\ref{lum} to vertices $u$ and $w$ with $\bv=\be_u-\frac{\alpha}{\beta}\be_w$ yields $\sqrt{n}\leq 1+\frac{|\alpha|}{|\beta|}$ and $\sqrt{n}\frac{|\alpha|}{|\beta|}\leq1+\frac{|\alpha|}{|\beta|}.$
If $\frac{|\alpha|}{|\beta|}\leq 1$, then the former equation yields $\sqrt{n}\leq 2$. If $\frac{|\alpha|}{|\beta|}>1$, then the latter equation yields $\sqrt{n}\leq 1+\frac{|\beta|}{|\alpha|}<2$, and so $\sqrt{n}<2$. In both cases, we get $n\leq 4$. Thus, if $n\geq 5$, then any weighting of almost all planar graphs yields a vertex that does not admit local $\epsilon$-uniform mixing. For the case of trees, adapting the proof of \cite[Theorem 9]{godsil2025quantum}, we get that almost all trees $T$ have $P_3$ as a limb (rooted at the degree two vertex). That is, almost all trees have the same form as the graph $X$ in Figure \ref{fig1}, where $H$ is a tree. Applying the same argument yields the same conclusion.
\end{proof}

Let $X$ be an unweighted graph on $n$ vertices. If $X$ is
$k$-cyclic, then $|E(X)|=n-1+k$, and so $\overline{d}=2+\frac{2(k-1)}{n}$. If $X$ is planar, then $|E(X)|\leq 3n-6$, and so $\overline{d}=\frac{2(3n-6)}{n}$. It is also known that if $X$ is a triangle-free planar graph, then $|E(X)|\leq 2n-4$ for all $n\geq 3$. Furthermore, it is shown in \cite{dowden2016extremal} that if $X$ is a $C_4$- or $C_5$-free planar graph, then $|E(X)|\leq \frac{15}{7}(n-2)$ for all $n\geq 4$ or $|E(X)|\leq \frac{12n-33}{5}$ for all $n\geq 11$, respectively. Using these facts, we prove the following result.

\begin{corollary}
\label{Cor:tree}
Let $M\in\{L,Q\}$ and $X$ be an unweighted graph.\hspace{-0.02in} If local $\epsilon$-uniform mixing occurs at $u$ then:
\begin{enumerate}
\item If $X$ is $k$-cylic, then $\operatorname{deg}u\leq 4+\frac{4(k-1)}{n}$. In particular, if $X$ is a tree, then $\operatorname{deg}u\leq 3$, while if $0\leq 4(k-1)<n$ (e.g., when $X$ is unicyclic), then $\operatorname{deg}u\leq 4$.
\item If $X$ is a planar graph, then $\operatorname{deg}u\leq 12-\frac{24}{n}$. In particular, $\operatorname{deg}u\leq 9$ whenever $n\leq 11$, $\operatorname{deg}u\leq 10$ whenever $12\leq n\leq 23$, and $\operatorname{deg}u\leq 11$ otherwise. 
\item If $X$ is a triangle-free planar graph with $n\geq 3$ vertices, then $\operatorname{deg}u\leq 8-\frac{16}{n}$. In particular, $\operatorname{deg}u\leq 5$ whenever $n\leq 7$, $\operatorname{deg}u\leq 6$ whenever $8\leq n\leq 15$, and $\operatorname{deg}u\leq 7$ otherwise.
\item If $X$ is a $C_4$-free planar graph  with $n\geq 4$ vertices, then $\operatorname{deg}u\leq \frac{60}{7n}(n-2)$. In particular, $\operatorname{deg}u\leq 6$ whenever $ n\leq 10$, $\operatorname{deg}u\leq 7$ whenever $11\leq n\leq 29$, and $\operatorname{deg}u\leq 8$ otherwise.
\item If $X$ is a $C_5$-free planar graph with $n\geq 11$ vertices, then $\operatorname{deg}u\leq \frac{4}{5n}(12n-33)$. In particular, $\operatorname{deg}u\leq 7$ whenever $n\leq 16$, $\operatorname{deg}u\leq 8$ whenever $17\leq n\leq 43$, and $\operatorname{deg}u\leq 9$ otherwise.
\end{enumerate}
\end{corollary}

\begin{proof}
If $X$ is a $k$-cyclic graph admitting local $\epsilon$-uniform mixing occurs at vertex $u$, then Theorem~\ref{Cor:avedeg} yields $\operatorname{deg}u\leq2\overline{d}=2\left(2+\frac{2(k-1)}{n}\right)$. This proves 1. Statements 2-5 follows using the same argument.
\end{proof}

The following result is immediate from Corollary~\ref{Cor:tree}.

\begin{corollary}
The following unweighted graphs do not admit $\epsilon$-uniform mixing relative to $L$ or $Q$.
\begin{enumerate}
\item A $k$-cyclic graph with maximum degree at least $d$, where $d=5+\frac{4(k-1)}{n}$ if $n$ divides $4(k-1)$ and $d=4+\lceil\frac{4(k-1)}{n}\rceil$ (e.g., a tree and unicyclic graph with maximum degree at least four and five, resp.).
\item A triangle-free, a $C_4$-free and a $C_5$-free unweighted planar graph with $n\geq 3$, $n\geq 4$ and $n\geq 11$ vertices, and maximum degree at least seven, eight and nine, respectively.
\end{enumerate}
\end{corollary}

Recall that $c_{j,\ell}$ denotes the number of common neighbors of vertices $j$ and $\ell$. It is known that a graph is $C_4$-free if and only if every pair of vertices has at most one common neighbor, i.e., $c_{j,\ell}\leq 1$ for all $j,\ell\in V(X)$. In what follows, we let $q$ be the number of pairs of vertices in $X$ that are distance two. 

\begin{theorem}
\label{Cor:avedeg1}
Let $M=A$ and $X$ be an unweighted $C_4$-free graph. If local $\epsilon$-uniform mixing occurs at $u$ in $X$, then $\operatorname{deg}u\leq \frac{2}{n}\left(|E(X)|+q\right)$. 
\end{theorem}

\begin{proof}
As $|\cos(\theta_j-\theta_{\ell})|\leq 1$ and $c_{j,\ell}\leq 1$ for all $j,\ell\in V(X)$, Corollary~\ref{Cor:eqs}(1) yields the desired result. 
\end{proof}

\begin{corollary}
\label{umfam}
Let $M=A$ and $X$ be an unweighted graph. If local $\epsilon$-uniform mixing occurs at $u$, then:
\begin{enumerate}
\item If $X$ is a tree, then $\operatorname{deg}u\leq  \frac{2}{n}(n-1+q)$.
\item Let $X$ be a unicyclic graph. If $X$ has no $C_4$, then $\operatorname{deg}u\leq 2+\frac{2q}{n}$. Otherwise, $\operatorname{deg}u\leq  2+\frac{2q+4}{n}$.
\item If $X$ is a $C_4$-free planar graph, then $\operatorname{deg}u\leq \frac{30(n-2)}{7n}+\frac{2q}{n}$. 
\end{enumerate}
\end{corollary}

\begin{proof}
If $X$ is a tree or a unicyclic graph with no $C_4$, then Theorem~\ref{Cor:avedeg1} yields 1 and the first statement in 2. Now, if $X$ is a unicyclic graph with a $C_4$, then there are exactly two pairs of vertices with two common neighbors. If vertex $u$ admits local uniform mixing, then Corollary~\ref{Cor:eqs}(1) combined with the facts $|\cos(\theta_j-\theta_{\ell})|\leq 1$ and $c_{j,\ell}\leq 1$ for all $j,\ell\in V(X)$ except for exactly two pairs, yields $\operatorname{deg}u\leq \frac{2}{n}(n+q+2)=2+\frac{2q+4}{n}$. This proves the second statement of 2. To prove 3, let $X$ be a $C_4$-free planar graph so that $|E(X)|\leq \frac{15}{7}(n-2)$. Applying Corollary~\ref{Cor:eqs}(1), we get $\operatorname{deg}u\leq \frac{2}{n}\left(|E(X)|+q\right)\leq \frac{2}{n}\left(\frac{15}{7}(n-2)+q\right)$.
\end{proof}

\section{Trees and unicyclic graphs}\label{sec:tu}

We further rule out $\epsilon$-uniform mixing in trees and unicyclic graphs. Throughout, we assume that $M=A$.

\begin{corollary}
\label{treeuni}
Let $X$ be an unweighted bipartite graph on an odd number of vertices. If $X$ has a pendent vertex $u$, then local $\epsilon$-uniform mixing does not occur  at $u$. In particular, if $X$ is a tree or a bipartite unicyclic graph that is not cycle, then $\epsilon$-uniform mixing does not occur at a pendent vertex.
\end{corollary}

\begin{proof}
The first statement follows from Corollary~\ref{nodd}. Since any tree and bipartite unicyclic graph that is not a cycle possess at least one pendent vertex, the second statement is immediate.
\end{proof}

We now prove a result analogous to Corollary~\ref{treeuni} when $X$ has an even number of vertices.

\begin{theorem}
\label{unic1}
Let $X$ be an unweighted bipartite unicyclic graph $n$ vertices that is not a cycle, where $n$ is even. If $X$ admits $\epsilon$-uniform mixing, then there is an even number of vertices in $X$ with degree $d\equiv 2,3$ (mod 4) if and only if $n\equiv 0$ (mod 4).
\end{theorem}

\begin{proof}
Let $X$ be an unweighted bipartite unicyclic graph that is not a cycle. Suppose $n$ is even and $X$ admits $\epsilon$-uniform mixing. Since $\epsilon$-uniform mixing occurs at each vertex of $X$, Corollary~\ref{nodd} applies to every vertex $u$ of $X$. Hence, there is an even number of vertices in $X$ with degree $d\equiv 2,3$ (mod 4) if and only if $\frac{1}{2}n\operatorname{deg}u$ is even for each vertex $u$ of $X$. Equivalently, either $n\equiv 0$ (mod 4), or each vertex in $X$ has even degree and $\frac{n}{2}$ is odd. However, since $X$ is not a cycle, $X$ contains a pendent vertex, and so the latter condition does not hold. Thus, it must be that $n\equiv 0$ (mod 4).
\end{proof}

For singular trees, we have the following result.

\begin{theorem}
\label{thm:1tree}
Let $T$ be a singular unweighted tree on $n$ vertices with parts $B_1$ and $B_2$. If $T$ admits local $\epsilon$-uniform mixing at $u\in B_1$ and $0\in\supp_{\be_u}$, then $n$ is a perfect square, $|B_1|\geq \sqrt{n}$, and $n$ and $|B_1|$ have the same parities.
\end{theorem}

\begin{proof}
If $T$ is a singular tree, then the nullspace for $A$ contains a basis of eigenvectors, each having all entries in $\{-1,0,1\}$ \cite{akbari20061}. Thus, if $0\in\supp_{\be_u}$, then we can find an eigenvector $\bv$ for 0 with entries from $\{-1,0,1\}$. Invoking  Corollary~\ref{cor:singularbip} yields the desired result.
\end{proof}

\begin{corollary}
\label{cor:singularbip1}
Let $n$ be odd and $T$ be an unweighted tree on $n$ vertices with parts $B_1$ and $B_2$. If local $\epsilon$-uniform mixing occurs at vertex $u\in B_1$ and $0\in\supp_{\be_u}$, then $n$ is a perfect square, $|B_1|\geq \sqrt{n}$, $|B_1|$ is odd, and local $\epsilon$-uniform mixing does not occur at any vertex $w$ of $B_2$ such that $0\in\supp_{\be_w}$.
\end{corollary}

\begin{proof}
If $n$ is odd, then $T$ is singular. As $0\in\supp_{\be_u}$, Theorem~\ref{thm:1tree} implies that $n$ is a perfect square, $|B_1|\geq \sqrt{n}$, and $|B_1|$ is odd (since $\sqrt{n}$ is). If local $\epsilon$-uniform mixing occurs at vertex $w$ in $B_2$ with $0\in\supp_{\be_w}$, then Theorem~\ref{thm:1tree} once again implies that $|B_2|$ is odd. Thus, $n=|B_1|+|B_2|$ is even, a contradiction.
\end{proof}

\begin{proposition}
\label{tree2}
If $T$ is an unweighted tree with no degree two vertex, then $T$ contains twin vertices.
\end{proposition}

\begin{proof}
If $u$ is a vertex in $T$ adjacent to a pendent vertex $v$, then $\operatorname{deg}u\geq 3$ because $X$ has no degree two vertex. Consequently, $v$ must have a twin in $T$.
\end{proof}

\begin{theorem}
\label{tree1}
Let $T$ be an unweighted tree on $n\geq 5$ vertices, where $n$ is even. If $T$ admits $\epsilon$-uniform mixing, then there is an odd number of vertices in $T$ with degree $d\equiv 2,3$ (mod 4).
\end{theorem}

\begin{proof}
Since $n$ is even, $|E(T)|=n-1$ is odd. As $\epsilon$-uniform mixing occurs in $T$, Corollary~\ref{nodd} applies to each $u\in V(T)$. That is, there is an even number of vertices in $T$ with degree $d\equiv 2,3$ (mod 4) if and only if $\frac{1}{2}n\operatorname{deg}u$ is odd for each $u\in V(T)$. The latter statement is equivalent to each vertex in $T$ has odd degree and $\frac{n}{2}$ is odd. If each vertex in $T$ has odd degree, then $T$ has no degree two vertex. Applying Proposition~\ref{tree2} and Corollary~\ref{twin}, $\epsilon$-uniform mixing does not occur in $T$, a contradiction. Thus, it must be that some vertex in $X$ has even degree, and so there is an odd number of vertices in $T$ with degree $d\equiv 2,3$ (mod 4).
\end{proof}

A \textit{caterpillar} is a tree whose deletion of all pendent vertices results in a path.

\begin{corollary}
\label{27}
Let $T$ be an unweighted caterpillar on an even number $n\geq 5$ of vertices. Then $T$ does not admit $\epsilon$-uniform mixing if $T$ has twins, or if $T$ has no twins but has an even number of pendent vertices.
\end{corollary}

\begin{proof}
If $T$ has twins, then Corollary~\ref{twin} yields the result. If $T$ has no twins, then each vertex in $T$ has degree one, two or three. Since $n$ is even and there is an even number of pendent vertices, applying Theorem~\ref{tree1} yields the conclusion.
\end{proof}

Given a tree $T$, we let $X(T)$ denote the tree obtained from $T$ by attaching a pendent vertex to each vertex of $T$. Note that $X(T)$ has twice the number of vertices of $T$. 

\begin{theorem}
\label{X(T)}
Let $T$ be an unweighted tree on $n$ vertices such that $\operatorname{deg}u\equiv 1,2$ (mod 4) for each vertex $u$ of $T$. For all $n\geq 2$, then $X(T)$ does not admit $\epsilon$-uniform mixing.
\end{theorem}

\begin{proof}
First, suppose $n$ is even. Note that each vertex $u$ of $T$ has degree $\operatorname{deg}u\equiv 2,3$ (mod 4) in $X(T)$ and there are exactly $n$ of them. By the contrapositive of Theorem~\ref{tree1}, we get the desired result. Now, suppose $n$ is odd. Then $X(T)$ has $2n\not\equiv 0$ (mod 4) vertices, and so conclusion follows from Theorem~\ref{thm:eumbip}(1).
\end{proof}

\begin{corollary}
\label{path}
For all $n\geq 3$, the unweighted path $P_n$ does not admit $\epsilon$-uniform mixing. Moreover, for all $n\geq 4$, $X(P_n)$ does not admit $\epsilon$-uniform mixing.
\end{corollary} 

\begin{proof}
If $n$ is even, then $P_n$ does not admit $\epsilon$-uniform mixing by Corollary~\ref{27}. The odd case follows from Corollary~\ref{treeuni}. This proves the first statement. The second is immediate from Theorem~\ref{X(T)}.
\end{proof}

Next, we provide families of trees with no twin vertices that do not admit $\epsilon$-uniform mixing. The result below follows from Corollary~\ref{cor:subd} using the fact that a tree is a $0$-cyclic graph.

\begin{theorem}
\label{thm:subdtree}
If $T$ is a weighted tree, then $S(T)$ does not admit $\epsilon$-uniform mixing.
\end{theorem}

We end with the following result.

\begin{corollary}
\label{k13}
$K_{1,3}$ is the only unweighted tree with no degree-two vertex that admits $\epsilon$-uniform mixing.
\end{corollary}

\begin{proof} 
$K_{1,3}$ is the only unweighted tree on $n\leq 4$ that admits $\epsilon$-uniform mixing. Now, if $n\geq 5$ and $T$ has no degree-two vertex, then applying Proposition~\ref{tree2} and Corollary~\ref{twin} yields no $\epsilon$-uniform mixing in $T$.
\end{proof}

\section{Open questions}
\label{sec:oq}

In this work, we derived necessary conditions for local $\epsilon$-uniform mixing and used them to rule out its existence in numerous classes of graphs, most of which are non-regular. We showed that almost all planar graphs and almost all trees  contain a vertex that does not admit local $\epsilon$-uniform mixing. This rarity result for motivates us to search for graphs that admit local $\epsilon$-uniform mixing, especially non-regular ones. 

If a vertex admits local $\epsilon$-uniform mixing, is the target state $\bmu$ unique, or is it possible to obtain two linearly independent target states at different times? If $\bmu$ has $\pm 1$ entries, then $\bmu$ will be a unique target state by the monogamy of perfect state transfer between real pure states. 

Finally, which graph operations preserve the local $\epsilon$-uniform mixing property of the underlying graphs? Is it possible to build graphs with local $\epsilon$-uniform mixing from those that do not possess this property?

\section*{Acknowledgements}
The author is supported in part by the Pacific Institute for the Mathematical Sciences through the PIMS-Simons postdoctoral fellowship.

\bibliographystyle{alpha}
\bibliography{mybibfile}
\end{document}